\documentclass[12pt,a4paper]{article}

\usepackage{amssymb, amsmath, amsthm}

\usepackage{algorithm}
\usepackage{algpseudocode}

\usepackage[cm]{fullpage}
\usepackage[english]{babel}
\usepackage[pdftex]{graphicx}
\usepackage{booktabs}
\usepackage{xcolor}

\usepackage{tikz}
\usetikzlibrary{arrows.meta, shapes.geometric, positioning, fit, backgrounds}

\usepackage{algorithm}
\usepackage{algpseudocode}

\usepackage{hyperref}
\usepackage{autonum}

\newtheorem{thm}{Theorem}
\newtheorem{lem}{Lemma}
\newtheorem{prop}{Proposition}
\newtheorem{rem}{Remark}
\newtheorem{cor}{Corollary}

\title{Parareal in time and spectral in space fast L1 quasilinear subdiffusion solver}
\author{
Josefa Caballero\thanks{Departamento de Matem\'aticas, Universidad de Las Palmas de Gran Canaria, Campus de Tafira Baja, $35017$ Las Palmas de Gran Canaria, Spain.}, \and
{\L}ukasz P{\l}ociniczak\thanks{Faculty of Pure and Applied Mathematics, Wroclaw University of Science and Technology, Wyb. Wyspia\'nskiego 27, 50-370 Wroc{\l}aw, Poland, \underline{corresponding author:} \texttt{lukasz.plociniczak@pwr.edu.pl}}, \vspace{4pt}\and
Kishin  Sadarangani\thanks{Departamento de Matem\'aticas, Universidad de Las Palmas de Gran Canaria, Campus de Tafira Baja, $35017$ Las Palmas de Gran Canaria, Spain.}
}
\date{}

\begin{document}
\maketitle

\begin{abstract}
	We consider the initial–boundary value problem for a quasilinear time‑fractional diffusion equation of order $0<\alpha\leq 1$, and develop a fully discrete solver combining the parareal algorithm in time with a L1 finite‑difference approximation of the Caputo derivative and a spectral Galerkin discretization in space. Our main contribution is the first rigorous convergence proof for the parareal–L1 scheme in this nonlinear subdiffusive setting. By constructing suitable energy norms and exploiting the orthogonality of the spectral basis, we establish that the parareal iterations converge exactly to the fully serial L1–spectral solution in a finite number of steps, with rates independent of the fractional exponent. The spectral spatial discretization yields exponential accuracy in space, while the parareal structure induces a clock speedup proportional to the number of processors, making the overall method highly efficient. Numerical experiments for both subdiffusive and classical diffusion problems confirm our theoretical estimates and demonstrate up to an order of magnitude reduction in computational time compared to the conventional sequential solver. We observe that the speedup of the parareal method increases linearly with the fine integrator degrees of freedom. 
\end{abstract}



\section{Introduction}
We consider the following nonlocal quasilinear partial differential equation, which models sub-diffusive processes in various contexts.
\begin{equation}\label{eqn:MainPDE}
	\partial^\alpha_t u = \nabla \cdot \left(D(x, t, u) \nabla u\right) + f(x,t,u), \quad x \in \Omega \subseteq \mathbb{R}^d, \quad t \in (0, T], \quad \alpha \in (0,1), 
\end{equation} 
with spatial dimension $d\in\mathbb{N}$, in an open domain $\Omega$, to the end time $T>0$. Both the diffusion coefficient $D(x,t,u)$ and the source $f(x,t,u)$ can depend on the solution, and we assume that they satisfy the usual regularity assumptions that grant the existence and uniqueness. Let $D\in C^2(\Omega, \mathbb{R}_+, \mathbb{R})$ and $f\in C^2(\Omega, \mathbb{R}_+, \mathbb{R})$. Furthermore, we require $D$ to be bounded from above and below, and we denote by $L>0$ the Lipschitz constant of the dependence of $f$ and $D$ on their third variable, that is,
\begin{equation}\label{eqn:Assumptions}
	0<D_-\leq D(x,t,u) \leq D_+, \quad |D(x,t,u)-D(x,t,v)| + |f(x,t,u) - f(x,t,v)| \leq L |u-v|.
\end{equation}
Under these assumptions, the problem is well-posed. Moreover, \cite{Zac12} established that \eqref{eqn:MainPDE} has a unique strong solution even when the conditions are relaxed. We further supplement \eqref{eqn:MainPDE} with the Dirichlet boundary condition and a chosen initial condition
\begin{equation}\label{eqn:MainPDEConditions}
	\begin{cases}
		u(x,0) = u_0(x), & x \in \Omega, \\
		u(x,t) = 0, & x \in \partial \Omega. \\
	\end{cases}
\end{equation}
We also assume that $u_0 \in H_0^1(\Omega) \cap H^2(\Omega)$. The temporal operator is the Caputo fractional derivative
\begin{equation}\label{eqn:Caputo}
	\partial^\alpha_t u(x,t) = \frac{1}{\Gamma(1-\alpha)} \int_0^t (t-s)^{-\alpha} u_t(x,s)ds, \quad 0<\alpha<1,
\end{equation}
which emergence naturally follows from the physical modeling of memory effects in various diffusive phenomena \cite{klafter2012fractional, metzler2000random}. 

Quasilinear subdiffusion equations are an interesting subject of study because of their ability to model a range of anomalous diffusion phenomena that arise in complex physical, biological, and material systems. Unlike classical diffusion, subdiffusion processes, characterized by slower than normal particle spreading, capture the dynamics observed in heterogeneous or disordered media, such as porous materials \cite{plociniczak2014approximation, plociniczak2015analytical, plociniczak2019derivation, pachepsky2000simulating, El20}, neurological tissues \cite{magin2010fractional}, polymers with memory effects \cite{metzler2000random, muller2011nonlinear}, protozoa migration \cite{alves2016transient}, single particle tracking in biophysics \cite{tabei2013intracellular, Sun17, wong2004anomalous}, plasma physics \cite{Del05}, astrophysics \cite{lawrence1993anomalous}, chemotaxis \cite{langlands2010fractional} and financial mathematics \cite{jacquier2020anomalous}. The study of \eqref{eqn:MainPDE} is also motivated by the rich mathematical structure it possesses, such as singularity in time derivatives, nonlocal structure, and the power-type decay of the solution (see, for example, \cite{allen2016parabolic, wittbold2021bounded, vergara2015optimal, akagi2019fractional, dipierro2019decay}). There is a need for specialized numerical techniques to ensure stability and accuracy when addressing nonsmooth initial data and long-term memory effects. Recent advances have focused on developing robust and efficient numerical methods, including the L1 scheme coupled with the finite element \cite{Jin19a, Mus18, plociniczak2022error,plociniczak2024fully, plociniczak2023linear} or the spectral scheme \cite{Lin07}, and fast convolution quadrature methods \cite{lopez2025convolution, cuesta2006convolution}.

The main purpose of this paper is to rigorously analyze the fully discrete numerical scheme used to solve \eqref{eqn:MainPDE} in which the time integration is done in parallel within the computational time. One of the well established ways of constructing such a scheme is via the \textit{parareal} algorithm introduced in 2001 by Lions, Maday, and Turinici \cite{lions2001resolution} (interestingly, the first idea of solving differential equations in parallel by the multiple shooting method dates back to 1964 \cite{nievergelt1964parallel}). The parareal method has gained significant attention as an efficient parallel-in-time algorithm designed to accelerate the solution of time-dependent differential equations \cite{gander2007analysis}. Traditional time-stepping methods often suffer from inherent sequential bottlenecks, which limit their scalability on modern parallel computing architectures. The parareal algorithm overcomes these limitations by decomposing the temporal domain into subintervals and iteratively correcting coarse approximations with fine, parallelizable integrators, reducing computational time while maintaining accuracy. Originally developed for linear and nonlinear evolution equations \cite{lions2001resolution}, the flexibility of the parareal has allowed its application to a broad class of problems, including control theory \cite{maday2002parareal}, quantum systems \cite{maday2007monotonic}, plasma physics \cite{reynolds2013analytic}, molecular dynamics \cite{baffico2002parallel}, and fractional diffusion models \cite{xu2015parareal, fu2019preconditioned, wu2017fast}. An overview of parallel-time integration methods, not only parareal, can be found in \cite{gander201550, ong2020applications}. The latest account on these matters can be found in \cite{gander2025time} where the parabolic versus hyperbolic equations are compared and contrasted. 

The results presented in this paper deal with the convergence analysis of the parareal method applied to the quasilinear subdiffusion equation \eqref{eqn:MainPDE}. In the literature, several papers rigorously addressed the convergence of the parareal in the sense that the discretisation parameters are kept fixed while the parareal iterations increase. The expected result states that the parareal method converges in a finite number of steps to the numerical solution obtained by the fine integrator. In \cite{gander2007analysis} the cases of linear ordinary and partial differential equations have been analyzed, and the authors provided convergence proofs and various numerical examples. These results were generalized in \cite{gander2008nonlinear} to systems of nonlinear ODEs. On the other hand, the parareal convergence in the general and abstract framework for PDEs was provided in \cite{bal2005convergence}. In PDEs various questions concerning the interplay between space and time discretization parameters with the parareal iterations arise. A common technique is to define the coarse and fine propagators to operate on different time grids with a fixed ration of steps. In the case of parabolic equations, it was shown in \cite{mathew2010analysis} that the parareal with backward Euler integrators is contractive with respect to iterations (authors also found the bound of the contraction constant to be close to $0.3$). This result was later carried over to some more advanced, but concrete, fine integrators in \cite{wu2015convergence} and recently to a general and robust case in \cite{yang2021robust}. Recently, some interesting results on the construction of optimal coarse integrators have been published in \cite{jin2025optimizing}. In addition, research on parallel-time integration methods for nonlocal differential equations has been recently initiated. The situation with such problems is much more difficult than in their local counterparts, since the whole history of the undergoing process affects the solution at the present time. Then it is not obvious how to build fine integrators. Some first approaches were done in \cite{xu2015parareal} (where the authors also considered the convergence of the scheme), \cite{wu2018parareal} where local in time integrators have been used to solve systems of fractional ODEs, and in \cite{fu2019preconditioned} where both the space and time nonlocality have been considered. Finally, in \cite{biala2018parallel} the authors investigated a semilinear subdiffusion equation and devised a parallel algorithm to solve it. 

The results of this paper can be summarized as follows:
\begin{itemize}
	\item We construct a parareal method based on the L1 scheme in time and the spectral Galerkin in space to facilitate \textbf{fast and accurate} computations;
	\item We use the fine integrator from \cite{fu2019preconditioned}, prove that it satisfies a natural Lipschitz condition, and find its \textbf{discretization error} for the realistic smoothness assumption;
	\item We prove the \textbf{convergence} of the parareal method in the sense that its error is bounded by a multiple of the error of the fine integrator. The multiplicative factor converges to a value, uniform in the iteration count, after a \textbf{finite number} of iterations;  
	\item Numerical computations verify that the \textbf{speedup} of the parareal method over the fine integrator \textbf{scales linearly} with the number of degrees of freedom reaching values of $10-15$ for $O(10^4)$ dof. Additionally, it requires much less allocated memory. 
\end{itemize}
To our knowledge, this is the first rigorous proof of the convergence of the parareal method for the quasilinear subdiffusion equation. 

In what follows, we will use $C>0$ to denote a generic constant that can depend on the given data, the solution of the continuous problem \eqref{eqn:MainPDE}, but not on the discretisation parameters or the parareal iteration count.

\section{Scheme construction}
We begin constructing our numerical method by first introducing the usual sequential scheme and then applying a parallelization to it.

\subsection{Fully discrete sequential scheme}
To derive the numerical scheme, we transform our problem \eqref{eqn:MainPDE} into a weak form by multiplying the equation by a test function $\chi \in H_0^1(\Omega)$ and integrating by parts.
\begin{equation}\label{eqn:MainPDEWeak}
	\left(\partial^\alpha_t u, \chi\right) + a(D(t,u); u, \chi) = (f(t,u), u, \chi), \quad \chi \in H_0^1(\Omega),
\end{equation}
where, as in common practice, we suppress the explicit dependence on $x$ for clarity since it is the integration variable in the usual $L^2(\Omega)$-inner product $(\cdot, \cdot)$. Also, we have introduced the bilinear form
\begin{equation}\label{eqn:AForm}
	a(w; u,v) := \int_\Omega w \, \nabla u \cdot \nabla v \,dx, \quad u, v \in H_0^1(\Omega), \quad w \in L^\infty(\Omega). 
\end{equation} 
The spatial (Galerkin) discretization of our PDE begins with choosing a suitable \emph{finite dimensional} subspace $V_N \subseteq H_0^1(\Omega)$ such that
\begin{equation}
	N := \dim V_N. 
\end{equation} 
Later, in the implementation, we will choose $V_N$ to be a space constructed with Legendre orthogonal polynomials which will lead to \emph{spectral scheme} in space. Note that if we were to choose $V_N$ to be spanned by the piecewise linear \emph{tent} functions, we would have arrived at the Finite Element Method (FEM). At this point, there is no need to specify, since the theory is general enough. 

By $u_N \in C^1((0,T]; V_N)$ denote the semi-discrete approximation of the solution of \eqref{eqn:MainPDEWeak} satisfying
\begin{equation}\label{eqn:MainPDEWeakSemi}
	\left(\partial^\alpha_t u_N, \chi\right) + a(D(t,u_N); u_N, \chi) = (f(t,u_N), u_N, \chi), \quad \chi \in V_N(\Omega),
\end{equation}
in which we average the solution over the finite-dimensional space $V_N$. Now, discretization in time proceeds using the L1 scheme (see \cite{del2025note}) to approximate the Caputo derivative by linear interpolation of the integrand. First, we introduce the (coarse) grid of points 
\begin{equation}
	T_n := n \Delta T, \quad \Delta T = \frac{T}{N_t}, \quad n = 0, 1, ..., N_t.
\end{equation}
Then, for any $y=y(t)$, the L1 scheme reads
\begin{equation}\label{eqn:L1Scheme}
	\begin{split}
		\partial^\alpha_t y(T_{n+1}) &\approx \frac{(\Delta T)^{-\alpha}}{\Gamma(2-\alpha)}\left(y(T_{n+1}) - b_{n} y(0) - \sum_{i = 1}^{n} \left(b_{n-i}-b_{n-i+1}\right) y(T_i)\right), \\ b_j &:= (j+1)^{1-\alpha} - j^{1-\alpha}. 
	\end{split}
\end{equation}
Now, we discretize our equation \eqref{eqn:MainPDEWeakSemi} in time in a semi-implicit manner, that is, we evaluate the diffusion coefficient and the source in a previous time step while keeping the gradient term in place at the present time. This yields the following system of equations \emph{linear} for the quasilinear problem \eqref{eqn:MainPDE} written in terms of a time-marching scheme.
\begin{equation}\label{eqn:NumericalSchemeSequential}
	\begin{split}
		(U_{n+1}, \chi) &+ \left(\Delta T\right)^\alpha \Gamma(2-\alpha) a(D(t_n, U_{n}); U_{n+1}, \chi) \\
		&= b_{n} \left(U_{0}, \chi\right) + \sum_{i = 1}^{n} \left(b_{n-i}-b_{n-i+1}\right) \left(U_{i}, \chi\right) + \left(\Delta T\right)^\alpha \Gamma(2-\alpha) (f(t_n, U_{n}), \chi), \quad \chi \in V_N, 
	\end{split}
\end{equation}
where $V_N \ni U_{n+1}(x) \approx u(x, t_{n+1})$ is the numerical approximation of the exact solution. Note that, for the brevity of the notation, we suppress the dependence on the number of degrees of freedom of the Galerkin scheme $N$, which remain fixed. This is, however, indicated with capital letters and justified by the fact that in the majority of our analysis, we will focus only on the time-marching and parareal iterations. Expanding both $U_{t}$ and $\chi$ from the above to the basis of $V_N$ one can infer that the matrix of the resulting system is symmetric. Therefore, in each time step, the linear system is \emph{uniquely solvable} and yields the numerical solution at time $T_{n+1}$. By $U_{0:n}$ we denote the juxtaposition of the solution at consecutive times (notation borrowed from \cite{xu2015parareal}), that is, $(U_{0}, U_{1}, ..., U_{n})$. The solution to \eqref{eqn:NumericalSchemeSequential} can then be denoted as
\begin{equation}\label{eqn:CoarsePropagator}
	U_{n+1} = \mathcal{G}_n\left(U_{0:n}\right).
\end{equation}
The operator $\mathcal{G}$ is called the (coarse) \emph{propagator} and codes the information that the solution at the next time step is a result of algebraic manipulations on the whole \emph{history} of the solution. This is a discrete analogue of the fact that the Caputo fractional derivative is a non-local operator. As we shall see, this has a profound impact on the construction of the parareal method. By the nature of the time-marching, the solution \eqref{eqn:CoarsePropagator} can only be obtained sequentially: we have to know all the previous steps to advance the solution further. Next, we describe how it can be used to parallelize the scheme. 

\subsection{Parareal algorithm}\label{sec:Parareal}
\begin{figure}
	\centering
	\begin{tikzpicture}
		\draw[->] (0,0) -- (16.5,0) node[right] {$t$};
		
		\foreach \x in {13,14,15} {
			\draw (\x,0.1) -- (\x,-0.1);
		}
		
		\foreach \n in {0,1,2,3,4} {
			\draw (\n*4,0.3) -- (\n*4,-0.3) node[below] {$T_{\n}$};
		}
		
		\foreach \r in {1,2,3} {
			\node[above] at (\r+12,0.1) {$t_{3,\r}$};
		}
	\end{tikzpicture}
	\caption{The coarse and the fine time grids.}
	\label{fig:TimeGrids}
\end{figure}
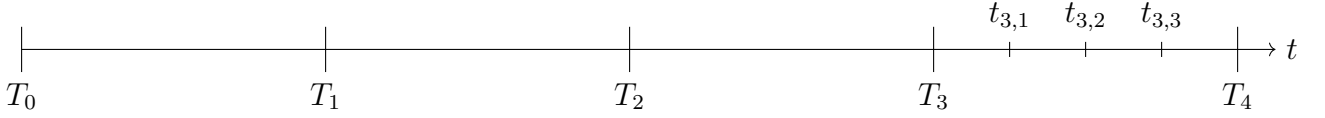

As we have seen, traditional time‐stepping schemes march forward sequentially, which does not make the utility of modern hardware that can work on many parallel threads at the same time. The parareal method overcomes this by coupling two propagators:
\begin{itemize}
	\item a \emph{coarse} solver \(\mathcal{G}\), which marches the solution only at the coarse time‐nodes
	$T_n = n\,\Delta T$, \(n=0,1,\dots,N_t\), with large step‐size \(\Delta T = T/N_t\),
	\item a \emph{fine} solver \(\mathcal{F}\), which works on a finer subgrid inside each coarse interval:
	\begin{equation}\label{eqn:FineGrid}
		t_{n,r} \;=\; T_n + r\,\Delta t,\quad r=0,1,\dots,M,\quad \Delta t = \frac{\Delta T}{M}, \quad M \in \mathbb{N},
	\end{equation}
	producing high‐accuracy solutions on each \([T_n,\,T_{n+1}]\) (see the graphical illustration in Fig. \ref{fig:TimeGrids}). 
\end{itemize}
From the above definition it is clear that $\mathcal{G}$ is computationally cheaper than the fine propagator $\mathcal{F}$. The main premise of the parareal algorithm is that, in practice, integrating the main equation sequentially on the whole interval $[0,T]$ with the fine propagator is prohibitively expensive. The idea is to evaluate $\mathcal{F}$ in \emph{parallel} on a number of independent processors at each of the coarse intervals $[T_n, T_{n+1}]$. Therefore, we simultaneously reduce the computational complexity of the problem and utilize parallelized hardware. 

The main difference between local and nonlocal equations is the fact that in the latter case we have to take into account the whole history of the solution. This might sound contradictory with the idea of parallelization: How do we apply the fine integrator without using the whole fine history? We will concretely answer to this question below, but now we just indicate that the fine propagator, which is suitable for parallel computations, can be defined as the one using some given coarse history $z_{0:n}$ and fine integrated values of the solution $y_{n, 0:r}$ in the last interval (see Fig. \ref{fig:TimeGrids})
\begin{equation}\label{eqn:FinePropagatorGeneral}
	y_{n, r+1} = \mathcal{F}_{n,r} \left(\text{present fine grained interval}; \text{history} \right)= \mathcal{F}_{n,r} \left(y_{n, 0:r}; z_{0:n}\right).
\end{equation}
In this way, we fine integrate only in the final interval and the whole history of the solution is given. The construction of the parareal algorithm in the classical (local) setting is straightforward. However, in a nonlocal setting, when history influences the present state of the solution, one has to be careful when designing the propagators. The easiest and optimal way to construct such a fine propagator is to use the L1 scheme to approximate the history $[0,T_n]$ on the coarse grid, while the solution in the last interval \([T_n, T_{n+1}]\) is discretized on the fine grid (this idea was introduced in \cite{fu2019preconditioned}). That is, the Caputo derivative of an arbitrary function \(y=y(t)\) can be approximated as follows
\begin{equation}
	\begin{split}
		\partial^\alpha_t y(t_{n,r}) 
		&= \frac{1}{\Gamma(1-\alpha)} \int_0^{t_{n,r}} (t_{n,r} - s)^{-\alpha} y'(s) ds \\
		&= \frac{1}{\Gamma(1-\alpha)} \left(\sum_{i=1}^n \int_{T_{i-1}}^{T_i} (t_{n,r} - s)^{-\alpha} y'(s) ds + \sum_{j=1}^r \int_{t_{n, j-1}}^{t_{n, j}} (t_{n,r} - s)^{-\alpha} y'(s) ds \right) \\
		&\approx \frac{1}{\Gamma(1-\alpha)} \left(\sum_{i=1}^n \frac{y(T_i)-y(T_{i-1})}{\Delta T} \int_{T_{i-1}}^{T_i} (t_{n,r} - s)^{-\alpha} ds \right.\\
		&\left.+ \sum_{j=1}^r \frac{y(t_{n,j}) - y(t_{n,j-1})}{\Delta t} \int_{t_{n, j-1}}^{t_{n, j}} (t_{n,r} - s)^{-\alpha} y'(s) ds \right),
	\end{split}
\end{equation}
where we have approximated $y=y(t)$ with a linear function on each of the subintervals. Evaluating integrals, we obtain
\begin{equation}\label{eqn:L1Fine}
	\begin{split}
		\partial^\alpha_t y(t_{n,r}) 
		&\approx \frac{1}{\Gamma(2-\alpha)} \left(\left(\Delta T\right)^{-\alpha} \sum_{i=1}^n b_{n-i+r/M} \left(y(T_i) - y(T_{i-1}) \right) + \left(\Delta t\right)^{-\alpha} \sum_{j=1}^r b_{r-j} \left(y(t_{n,j}) - y(t_{n, j-1}) \right) \right) \\
		&=\frac{1}{\Gamma(2-\alpha)} \left[\left(\Delta T\right)^{-\alpha} \left(b_{r/M}y(T_n)-b_{n-1+r/M} y(0)  - \sum_{i=1}^{n-1} \left(b_{n-i-1+r/M} - b_{n-i+r/M}\right) y(T_i) \right) \right.\\
		&\left.+ \left(\Delta t\right)^{-\alpha} \left(b_{0,}y(t_{n,r})-b_{r-1} y(t_{n,0})  - \sum_{j=1}^{r-1} \left(b_{r-j-1} - b_{r-j}\right) y(t_{n,j}) \right) \right] =: \delta^\alpha y(t_n),
	\end{split}
\end{equation}
where the symbol $b_j$ was defined in \eqref{eqn:L1Scheme}. All other terms in equation \eqref{eqn:MainPDEWeakSemi} remain exactly the same as in the coarse propagator scheme \eqref{eqn:CoarsePropagator}. That is, the fine grained approximation of $u(x, t)$ at time $t = t_{n,r}$ is the solution of the following system for $n \geq 0$ and $r = 0, 1, ..., M$ 
\begin{equation}\label{eqn:NumericalSchemeFine}
	\begin{split}
		(U_{n, r}, \chi) &+ \left(\Delta t\right)^\alpha \Gamma(2-\alpha) a(D(t_{n,r-1}, U_{n,r-1}); U_{n, r}, \chi) \\
		&= b_{r-1} \left(U_{n,0}, \chi\right) + \sum_{i = 1}^{r-1} \left(b_{r-j-1}-b_{r-j}\right) \left(U_{n,j}, \chi\right) + \left(\Delta t\right)^\alpha \Gamma(2-\alpha) (f(t_{n,r-1}, U_{n, r-1}), \chi), \\
		&+M^{-\alpha}\left(-b_{r/M} \left(U_{n}, \chi\right) + b_{n-1+r/M} \left(U_{0}, \chi\right) + \sum_{i=1}^{n-1} \left(b_{n-i-1+r/M} - b_{n-i+r/M}\right) \left(U_{i}, \chi\right) \right), \quad \chi \in V_N, 
	\end{split}
\end{equation}
where the last part is the coarse history. Of course, $U_{n,0} = U_n$. Similarly as above, the solution of the following system $U_{n,r}$ defines the fine integrator
\begin{equation}\label{eqn:FinePropagator}
	U_{n, r} = \mathcal{F}_{n, r-1} \left(U_{n,0:r-1}; U_{0:n}\right).
\end{equation}

The parareal is based on iterating parallel fine-propagator computations and sequentially correcting the numerical approximation via the coarse propagator. By $U^k_n$ denote the parareal approximation of the solution to \eqref{eqn:MainPDE} after $k$ iterations. Therefore, approximations $V_N \ni U_n^k(x) \approx u(x, T_n)$ after $k$ parareal iterations are sought. The parareal algorithm proceeds in two steps. 

\medskip
\noindent\textbf{1. Initialization through the coarse propagator.}  
Starting from \(U_0^0 = u_0\) perform a single sequential sweep over the coarse grid.
\begin{equation}
	U_{n+1}^0 = \mathcal{G}_n\left(U^0_{0:n}\right), \quad n = 0,1,\dots,N-1.
\end{equation}
This yields a rough and cheap prediction of the solution in each coarse time slice \(T_n\).

\medskip
\noindent\textbf{2. Iterative correction with fine solvers.}  
For each iteration \(k = 0,1,\dots,K-1\), set \(U_0^{k+1} = u_0\) and then in \emph{parallel} for all \(n=0,\dots,N-1\) compute:

\begin{enumerate}
	\item[a)] A fine‐grid propagation on each \([T_n, T_{n+1}]\):
	\begin{equation}
		\widetilde U_{n+1} = \mathcal{F}_{n, M}\left(U^k_{n, 0:M}; U^k_{0:n}\right),
	\end{equation}
	which in practice integrates step by step through the points \(t_{n,1},t_{n,2},\dots,t_{n,M}\) and uses the history of the coarse propagator. 
	\item[b)] A corrective update on the coarse nodes:
	\begin{equation}\label{eqn:Parareal}
		U_{n+1}^{k+1} = \mathcal{G}\left(U_n^{k+1}\right) + \left(\widetilde U_{n+1} - \mathcal{G}\left(U_n^k\right)\right).
	\end{equation}
\end{enumerate}
Suppose that for $k\rightarrow\infty$ (we prove in the following that convergence is actually obtained after finite number of iterations) the parareal approximation converges to some $U^*_n$, then by \eqref{eqn:Parareal} we have
\begin{equation}\label{eqn:PararealConvergence}
	U^*_{n+1} = \mathcal{G}\left(U_n^{*}\right) + \left(\mathcal{F}_{n, M}\left(U^*_{n, 0:M}; U^*_{0:n}\right)- \mathcal{G}\left(U_n^*\right)\right),
\end{equation}
that is, \(U^*_{n+1} = \mathcal{F}_{n, M}\left(U^*_{n, 0:M}; U^*_{0:n}\right)\). Hence, the limit of iterations is the fine numerical solution of the problem. In practice, a sufficient accuracy can be obtained after just a few iterations, in which the most expensive computations can be done in parallel. This is the main advantage of using the parareal method for which the flowchart is given in Fig. \ref{fig:parareal}.


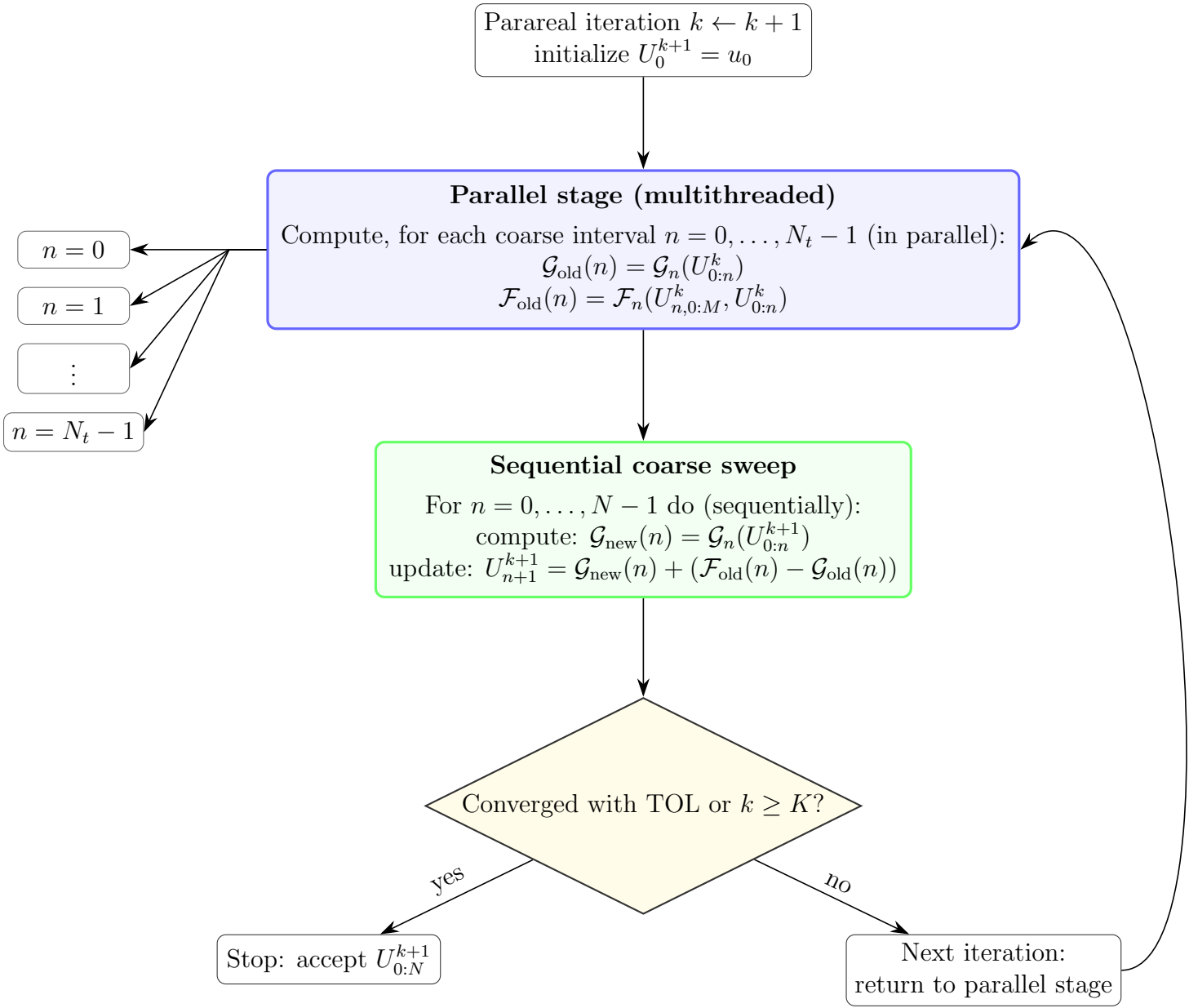
\begin{figure}
	\centering
	\begin{tikzpicture}[
		node distance = 10mm,
		box/.style = {rectangle, draw=black!80, rounded corners, minimum width=36mm, minimum height=8mm, align=center, fill=white},
		parbox/.style = {rectangle, draw=blue!60, very thick, rounded corners, inner sep=6pt, align=center, fill=blue!5},
		seqbox/.style = {rectangle, draw=green!60, very thick, rounded corners, inner sep=6pt, align=center, fill=green!5},
		small/.style = {rectangle, draw=black!60, rounded corners, minimum width=18mm, minimum height=6mm, align=center, fill=white},
		decision/.style = {diamond, draw=black!80, thick, aspect=2, inner sep=2pt, align=center, fill=yellow!10},
		arrow/.style = {-{Stealth[length=3mm,width=2mm]}, semithick},
		parallelbrace/.style = {decorate,decoration={brace,amplitude=6pt,mirror,raise=4pt},thick}
		]
		
		\node[box] (start) {Parareal iteration $k \leftarrow k+1$\\ initialize $U^{k+1}_0 = u_0$};
		
		\node[parbox, below=15mm of start] (parallel) {
			\textbf{Parallel stage (multithreaded)}\\[4pt]
			Compute, for each coarse interval $n=0,\dots,N_t-1$ (in parallel):\\
			$\mathcal{G}_{\text{old}}(n) = \mathcal{G}_n(U^k_{0:n})$\\
			$\mathcal{F}_{\text{old}}(n) = \mathcal{F}_n(U^k_{n,0:M}, U^k_{0:n})$
		};
		
		\node[small, left=22mm of parallel.west, anchor=east] (task0) {$n=0$};
		\node[small, below=3mm of task0] (task1) {$n=1$};
		\node[small, below=3mm of task1] (task2) {$\vdots$};
		\node[small, below=3mm of task2] (task3) {$n=N_t-1$};
		\draw[arrow] (parallel.west) -- ++(-6mm,0) coordinate (pstart) -- (task0.east);
		\draw[arrow] (parallel.west) -- ++(-6mm,0) -- (task1.east);
		\draw[arrow] (parallel.west) -- ++(-6mm,0) -- (task2.east);
		\draw[arrow] (parallel.west) -- ++(-6mm,0) -- (task3.east);
		
		\node[seqbox, below=18mm of parallel] (sequential) {
			\textbf{Sequential coarse sweep}\\[4pt]
			For $n=0,\dots,N-1$ do (sequentially):\\
			compute: $\mathcal{G}_{\text{new}}(n) = \mathcal{G}_n(U^{k+1}_{0:n})$\\
			update: $U^{k+1}_{n+1} = \mathcal{G}_{\text{new}}(n) + (\mathcal{F}_{\text{old}}(n) - \mathcal{G}_{\text{old}}(n))$
		};
		
		\node[decision, below=16mm of sequential] (decide) {Converged with TOL or $k\ge K$?};
		
		\node[box, below left=12mm and 15mm of decide] (end) {Stop: accept $U^{k+1}_{0:N}$};
		\node[box, below right=12mm and 15mm of decide] (loop) {Next iteration: \\ return to parallel stage};
		
		\draw[arrow] (start) -- (parallel);
		\draw[arrow] (parallel) -- (sequential);
		\draw[arrow] (sequential) -- (decide);
		\draw[arrow] (decide) -- node[midway, above, sloped] {yes} (end);
		\draw[arrow] (decide) -- node[midway, above, sloped] {no} (loop);
		
		\draw[arrow] (loop.east) .. controls +(+22mm,0) and +(+22mm,20mm) .. (parallel.east);
		
	\end{tikzpicture}
	\caption{The parareal algorithm.}
	\label{fig:parareal}
\end{figure}

\section{Analysis}
Now, after constructing all the necessary methods and algorithms, we can proceed to proving their properties. 

\subsection{Properties of integrators}
First, we claim that both the coarse and fine propagators are \(L^2\)-Lipschitz (the \(L^\infty\)-Lipschitz property is straightforward).  

\begin{prop}\label{prop:Lipschitz}
	Let $\left\{V_n\right\}_n$, $\left\{W_n\right\}_n$, and $\left\{Z_{n,r}^{(1,2)}\right\}_{n,r}$ be sequences of functions from $V_N$. Then, 
	\begin{equation}
		\begin{split}
			\|\mathcal{G}_n\left(V_{0:n}\right) - \mathcal{G}_n\left(W_{0:n}\right)\| 
			&\leq \sqrt{\frac{1+C (\Delta T)^\alpha \Gamma(2-\alpha)}{1-L_f \left(\Delta T\right)^\alpha \Gamma(2-\alpha)}} \max_{0\leq i\leq n}\|V_i - W_i\|, \\
			\left\|\mathcal{F}_{n,r}\left(Z^{(1)}_{n,0:r}; V_{0:n}\right) - \mathcal{F}_{n,r}\left(Z^{(2)}_{n,0:r}; W_{0:n}\right)\right\| 
			&\leq \frac{1}{\sqrt{1-L_f (\Delta t)^\alpha \Gamma(2-\alpha)}} \\
			\times\left(\sqrt{1+ C(\Delta t)^\alpha \Gamma(2-\alpha)} \max_{0\leq i\leq r-1} \|Z^{(1)}_{n,i} - Z^{(2)}_{n,i}\| \right. & \left.+ \sqrt{2 b_{r/M}} M^{-\alpha} \max_{0\leq j\leq n} \|V_j - W_j\|\right),
		\end{split}
	\end{equation}
	where the constant $C>0$ depends on $D$, $f$ and not on $\alpha$. 
\end{prop}
\begin{proof}
	Take any sequences $\{V_n\}_{n\in N}$, $\{W_n\}_{n\in N}$ that belong to $V_N$. From the definition of the coarse propagator \eqref{eqn:CoarsePropagator} we know that $V_{n+1} = \mathcal{G}_n\left(V_{0:n}\right)$ and similarly for $W_n$. Put $V_n$ and $W_n$ in \eqref{eqn:NumericalSchemeSequential} and subtract to obtain
	\begin{equation}
		\begin{split}
			\left(\mathcal{G}_n\left(V_{0:n}\right) - \mathcal{G}_n\left(W_{0:n}\right), \chi\right) &+ \left(\Delta T\right)^\alpha \Gamma(2-\alpha) \left(a(D(t_n, V_{n}); V_{n+1} - W_{n+1}, \chi) \right.\\
			&\left.+ a(D(t_n, V_{n})-D(t_n, W_{n}); W_{n+1}, \chi)\right)\\
			&= b_{n} \left(V_{0}-W_{0}, \chi\right) + \sum_{i = 1}^{n} \left(b_{n-i}-b_{n-i+1}\right) \left(V_{i}-W_{i}, \chi\right) \\
			&+ \left(\Delta T\right)^\alpha \Gamma(2-\alpha) (f(t_n, V_{n}) - f(t_n, W_{n}), \chi), \quad \chi \in V_N.
		\end{split}
	\end{equation}
	Now, choose $\chi = \mathcal{G}_n\left(V_{0:n}\right) - \mathcal{G}_n\left(W_{0:n}\right)$. Then, by using the boundedness of $D$ and the definition of the $a$-form \eqref{eqn:AForm}, the left-hand side of the above becomes
	\begin{equation}
		\begin{split}
			\text{Left } &\geq \|\mathcal{G}_n\left(V_{0:n}\right) - \mathcal{G}_n\left(W_{0:n}\right)\|^2 + \left(\Delta T\right)^\alpha \Gamma(2-\alpha) D_- \|\nabla\left(\mathcal{G}_n\left(V_{0:n}\right) - \mathcal{G}_n\left(W_{0:n}\right)\right)\|^2 \\
			&- L_D\left(\Delta T\right)^\alpha \Gamma(2-\alpha)\|V_n - W_n\| \|\nabla W_{n+1}\|_\infty \|\nabla\left(\mathcal{G}_n\left(V_{0:n}\right) - \mathcal{G}_n\left(W_{0:n}\right)\right)\|.
		\end{split}
	\end{equation}
	Now, for the last term, we use the boundedness of $\|\nabla W_{n+1}\|_\infty \leq W$ and the $\epsilon$-Cauchy inequality $ab \leq (\epsilon^{-1} a^2 + \epsilon^{1} b^2)/2$ with a suitable $\epsilon = 2 D_-/(L_D W)$ for $a=\|V_n - W_n\|_\infty$ and $b = \|\nabla\left(\mathcal{G}_n\left(V_{0:n}\right) - \mathcal{G}_n\left(W_{0:n}\right)\right)\|$, we further have
	\begin{equation}\label{eqn:LipschitzLeft}
		\begin{split}
			\text{Left } &\geq \|\mathcal{G}_n\left(V_{0:n}\right) - \mathcal{G}_n\left(W_{0:n}\right)\|^2 + \left(\Delta T\right)^\alpha \Gamma(2-\alpha) \left(D_- \|\nabla\left(\mathcal{G}_n\left(V_{0:n}\right) - \mathcal{G}_n\left(W_{0:n}\right)\right)\|^2 \right. \\
			&\left.- C\|V_n - W_n\|^2 - D_-\|\nabla\left(\mathcal{G}_n\left(V_{0:n}\right) - \mathcal{G}_n\left(W_{0:n}\right)\right)\|^2 \right) \\
			&= \|\mathcal{G}_n\left(V_{0:n}\right) - \mathcal{G}_n\left(W_{0:n}\right)\|^2 - C\left(\Delta T\right)^\alpha \Gamma(2-\alpha)\|V_n - W_n\|^2,
		\end{split}
	\end{equation}
	with a suitable constant $C>0$. On the other hand, the right-hand side, by Cauchy inequality $a b \leq (a^2+b^2)/2$ and the Lipschitz condition on $f$, becomes
	\begin{equation}
		\begin{split}
			\text{Right } 
			&\leq \frac{1}{2}\left(b_n \|V_{0}-W_{0}\|^2 + \sum_{i=1}^n \left(b_{n-i} - b_{n-i+1}\right)\|V_{i}-W_{i}\|^2\right) + \frac{1}{2} \|\mathcal{G}_n\left(V_{0:n}\right) - \mathcal{G}_n\left(W_{0:n}\right)\|^2 \\
			&+ \frac{1}{2}L_f \left(\Delta T\right)^\alpha \Gamma(2-\alpha) \left(\|V_n - W_n\|^2 + \|\mathcal{G}_n\left(V_{0:n}\right) - \mathcal{G}_n\left(W_{0:n}\right)\|^2\right),
		\end{split}
	\end{equation}
	which, after taking maximum over the previous times is
	\begin{equation}
		\begin{split}
			\text{Right } 
			&\leq \frac{1}{2}\left(b_n + \sum_{i=1}^n \left(b_{n-i} - b_{n-i+1}\right)\right)\max_{0\leq i \leq n}\|V_{i}-W_{i}\|^2 + \frac{1}{2} \|\mathcal{G}_n\left(V_{0:n}\right) - \mathcal{G}_n\left(W_{0:n}\right)\|^2 \\
			&+ \frac{1}{2}L_f \left(\Delta T\right)^\alpha \Gamma(2-\alpha) \left(\|V_n - W_n\|^2 + \|\mathcal{G}_n\left(V_{0:n}\right) - \mathcal{G}_n\left(W_{0:n}\right)\|^2\right).
		\end{split}
	\end{equation}
	And since the series is telescoping to $b_0 = 1$, we have
	\begin{equation}\label{eqn:LipschitzRight}
		\begin{split}
			\text{Right } 
			&\leq \frac{1}{2}\max_{0\leq i \leq n}\|V_{i}-W_{i}\|^2 + \frac{1}{2} \|\mathcal{G}_n\left(V_{0:n}\right) - \mathcal{G}_n\left(W_{0:n}\right)\|^2 \\
			&+ \frac{1}{2}L_f \left(\Delta T\right)^\alpha \Gamma(2-\alpha) \left(\max_{0\leq i \leq n}\|V_i - W_i\|^2 + \|\mathcal{G}_n\left(V_{0:n}\right) - \mathcal{G}_n\left(W_{0:n}\right)\|^2\right).
		\end{split}
	\end{equation}
	Now, combining the left \eqref{eqn:LipschitzLeft} and right \eqref{eqn:LipschitzRight} estimates we arrive at
	\begin{equation}
		\begin{split}
			\frac{1}{2}\|\mathcal{G}_n\left(V_{0:n}\right) - \mathcal{G}_n\left(W_{0:n}\right)\|^2 &\leq \frac{1}{2}\left(1+C (\Delta T)^\alpha \Gamma(2-\alpha)\right) \max_{0\leq i\leq n}\|V_i - W_i\|^2 \\
			&+  \frac{1}{2}L_f \left(\Delta T\right)^\alpha \Gamma(2-\alpha) \|\mathcal{G}_n\left(V_{0:n}\right) - \mathcal{G}_n\left(W_{0:n}\right)\|^2,
		\end{split}
	\end{equation}
	where we have renamed the constant $C>0$. Simplifying, we obtain
	\begin{equation}
		\|\mathcal{G}_n\left(V_{0:n}\right) - \mathcal{G}_n\left(W_{0:n}\right)\|^2 \leq \frac{1+C (\Delta T)^\alpha \Gamma(2-\alpha)}{1-L_f \left(\Delta T\right)^\alpha \Gamma(2-\alpha)} \max_{0\leq i\leq n}\|V_i - W_i\|^2,
	\end{equation}
	which completes the proof of $\mathcal{G}$ being Lipschitz. The proof for $\mathcal{F}$ proceeds in a completely analogous way by using the $\epsilon$-Cauchy and Schwarz inequalities with the additional term representing history. 
\end{proof}

\begin{rem}
	It follows from the proof, that for the linear diffusion, that is when $D=D(x,t)$ and $f=f(x,t)$, we can take $C = L_f = 0$.
\end{rem}

\begin{rem}
	When $\Delta T \rightarrow 0^+$ we asymptotically have
	\begin{equation}
		\sqrt{\frac{1+C (\Delta T)^\alpha \Gamma(2-\alpha)}{1-L_f \left(\Delta T\right)^\alpha \Gamma(2-\alpha)}} \sim 1 + \frac{\Gamma(2-\alpha)}{2} (C+L_f) \left(\Delta T\right)^\alpha,
	\end{equation}
	that is, the coarse L1 integrator satisfies the usual Lipschitz condition with the $(\Delta T)^\alpha$ term (compare \cite{xu2015parareal}, where authors assume that the $L^\infty$ constant is $1+C \Delta t$, which is not valid for the discretisation of the Caputo derivative). 
\end{rem}

Now, since the fine integrator \eqref{eqn:NumericalSchemeFine} is constructed in a non-standard way, we have to find its discretisation error. This, in turn, strongly depends on the regularity of the solution. One shows that a solution of the linear subdiffusion equation satisfies \cite{Sak11}
\begin{equation}
	u(x,t) = u_0(x) + t^\alpha \phi(x,t), 
\end{equation}
where $\phi$ is smooth up to $t=0$. That is, the typical solution is continuous near the origin with a derivative that blows up as $t^{\alpha-1}$ (see relevant results for diffusion nonlocal in space, in \cite{del2025numerical}). To simplify notation and since the space variable is not important in the following, assume that we have a function $y=y(t)$ satisfying
\begin{equation}
	y(t) = g(t) + t^\alpha k(t),
\end{equation}
where $g$ and $k$ are Lipschitz continuous on $[0,T]$. Note that then for $t>s$ we have
\begin{equation}
	|y(t) - y(s)| \leq |g(t) - g(s)| + \left(t^\alpha-s^\alpha\right) |k(t)| + s^\alpha |k(t) - g(s)|,
\end{equation}
now by the elementary inequality (see \cite{del2025numerical}, Appendix B) for $a>b>0$ stating that $a^\alpha-b^\alpha \leq a^{\alpha-1}(a-b)$, we obtain
\begin{equation}\label{eqn:DiffusionRegularity}
	|y(t) - y(s)| \leq (C+T^\alpha) (t-s) + t^{\alpha-1} (t-s) = C (1+t^{\alpha-1})(t-s).
\end{equation}
In \cite{del2025numerical} it has been proved that precisely this regularity is enjoyed by a general class of solutions of the space-time nonlocal diffusion equation under weak assumptions on the data.

\begin{prop}\label{prop:L1Hybrid}
	Let $y = y(t)$ be a $C[0,T]$ function satisfying the regularity assumption \eqref{eqn:DiffusionRegularity}. Then, with $\delta^\alpha$ defined in \eqref{eqn:L1Fine}, we have
	\begin{equation}
		|\partial^\alpha_t y(t_{n,r}) - \delta^\alpha y(t_{n,r})| \leq C
		\begin{cases}
			\begin{cases} 1, & r = 1, \\ (\Delta t)^{1-\alpha} t_{0,r}^{\alpha-1}, & r\geq 2\end{cases}, & n = 0, \vspace{2pt}\\
			\dfrac{M^\alpha}{\Gamma(1-\alpha)}, & n = 1, \vspace{2pt}\\
			(\Delta t)^{1-\alpha} T_n^{\alpha-1} M^{\alpha+1}, & n \geq 2,
		\end{cases}
	\end{equation}
	where $C$ depends only on the function $y$.
\end{prop}
\begin{proof}
	Let $I_{\Delta T}$ and $I_{\Delta t}$ be the interpolation operator on the $\Delta T$ and $\Delta t$ grids, respectively. It is easy to obtain the classical interpolation error estimate
	\begin{equation}\label{eqn:InterpolationError}
		|I_{\Delta T} y(t) - y(t)| \leq \frac{1}{\Delta T} \left|(t-T_i) (y(T_{i+1})-y(t)) + (T_{i+1}-t)(y(T_i)-y(t))\right|, \quad t \in [T_i, T_{i+1}], \quad 0\leq i \leq N-1, 
	\end{equation} 
	and similarly for $I_{\Delta t}$. Now, with a suitable integration by parts in the definition of the Caputo derivative \eqref{eqn:Caputo} we can show that
	\begin{equation}
		\partial^{\alpha}_t y(t) =\frac{1}{\Gamma(1-\alpha)}\frac{y(t)-y(0)}{t^{\alpha}}+\frac{\alpha}{\Gamma(1-\alpha)}\int_0^t \frac{y(t)-y(s)}{(t-s)^{1+\alpha}} ds.
	\end{equation}
	Since, by construction, the L1 scheme is the Caputo derivative applied to the linear interpolant we have
	\begin{equation}
		|\partial^\alpha_t y(t_{n,r}) - \delta^\alpha y(t_{n,r})| \leq \frac{\alpha}{\Gamma(1-\alpha)} \left[\sum_{i=0}^{n-1} \int_{T_i}^{T_{i+1}} \frac{|I_{\Delta T}y(s) - y(s)|}{(t_{n,r}-s)^{1+\alpha}} ds + \sum_{j=0}^{r-1} \int_{t_{n,j}}^{t_{n,j+1}} \frac{|I_{\Delta t}y(s) - y(s)|}{(t_{n,r}-s)^{1+\alpha}} ds\right],
	\end{equation}
	where, as in the construction of the fine integrator \eqref{eqn:L1Fine}, we have separated the history part and the current interval. Now, we have to bound each term. First, by the interpolation error estimate \eqref{eqn:InterpolationError} and the regularity assumption \eqref{eqn:DiffusionRegularity}, we obtain the following
	\begin{equation}\label{eqn:TruncationGeneralError}
		\begin{split}
			|\partial^\alpha_t y(t_{n,r}) - \delta^\alpha y(t_{n,r})| 
			&\leq C\frac{\alpha}{\Gamma(1-\alpha)} \left[(\Delta T)^{-1}\sum_{i=0}^{n-1} \int_{T_i}^{T_{i+1}} (1+s^{\alpha-1})(s-T_i)(T_{i+1}-s)(t_{n,r}-s)^{-1-\alpha} ds \right. \\
			&\left.+ (\Delta t)^{-1}\sum_{j=0}^{r-1} \int_{t_{n,j}}^{t_{n,j+1}} (1+s^{\alpha-1})(s-t_{n,j})(t_{n,j+1}-s)(t_{n,r}-s)^{-1-\alpha} ds\right],
		\end{split}
	\end{equation}
	since $T_{i+1}^{\alpha-1} \leq s^{\alpha-1}$ on the respective interval (and similarly for $t_{n,r}$). 
	
	\medskip\noindent\textbf{The cases $\mathbf{n=0,1}$.} Notice that for $n=0$ the operator $\delta^\alpha$ is the same as the $\Delta t$-L1 scheme in the interval $[0, T_1]$ and the estimate of its error has been obtained in \cite{del2025numerical}. Next, when $n=1$ the history part consists of only the first coarse interval $[0, T_1]$, hence
	\begin{equation}\label{eqn:TruncationN1}
		\begin{split}
			|\partial^\alpha_t y(t_{1,r}) - \delta^\alpha y(t_{1,r})| 
			&\leq C\frac{\alpha}{\Gamma(1-\alpha)} \left[(\Delta T)^{-1}\int_0^{T_1} (1+s^{\alpha-1})s(T_1-s)(t_{1,r}-s)^{-1-\alpha} ds \right.\\
			&\left.+ (\Delta t)^{-1}\sum_{j=0}^{r-1} \int_{t_{1,j}}^{t_{1,j+1}} (1+s^{\alpha-1})(s-t_{1,j})(t_{1,j+1}-s)(t_{1,r}-s)^{-1-\alpha} ds \right]
		\end{split}
	\end{equation}
	The first integral is bounded since $s^{\alpha-1} > 1$ and
	\begin{equation}
		\begin{split}
			(\Delta T)^{-1}&\int_0^{T_1} (1+s^{\alpha-1})s(T_1-s)(t_{1,r}-s)^{-1-\alpha} ds \\
			&\leq 2 \int_0^{T_1} s^\alpha (t_{1,r} - s)^{-1-\alpha} ds \leq \frac{2}{\alpha} (\Delta T)^{\alpha} \left((r \Delta t)^{-\alpha} - \left(\Delta T + r \Delta t\right)^{-\alpha}\right) \leq \frac{2}{\alpha} M^{\alpha} (\Delta t)^{\alpha} t_{0,r}^{-\alpha}.
		\end{split}
	\end{equation}
	For the second term, the sum in \eqref{eqn:TruncationN1}, we first consider the case with $r=1$, for which we have
	\begin{equation}
		\begin{split}
			(\Delta t)^{-1}&\int_{\Delta T}^{\Delta T + \Delta t} (1+s^{\alpha-1})(s-\Delta T)(\Delta T + \Delta t-s)(\Delta T + \Delta t-s)^{-1-\alpha} ds \\
			&\leq 2 \int_{\Delta T}^{\Delta T + \Delta t} s^{\alpha-1}(\Delta T + \Delta t-s)^{-\alpha} ds \leq \frac{2}{1-\alpha} (\Delta T)^{\alpha-1} (\Delta t)^{1-\alpha} = \frac{2 M^{\alpha-1}}{1-\alpha}. 
		\end{split}
	\end{equation}
	Now, if $1 < r \leq M$, we further estimate by splitting the last term in the sum in \eqref{eqn:TruncationN1}. The first $r-1$ integrals can be joined after bounding $(s-t_{n,j})(t_{n,j+1}-s) \leq (\Delta t)^2$
	\begin{equation}
		\begin{split}
			(\Delta t)^{-1}&\sum_{j=0}^{r-1} \int_{t_{1,j}}^{t_{1,j+1}} (1+s^{\alpha-1})(s-t_{1,j})(t_{1,j+1}-s)(t_{1,r}-s)^{-1-\alpha} ds \\
			&\leq \Delta t \int_{t_{1,0}}^{t_{1,r-1}} (1+s^{\alpha-1}) (t_{1,r} - s)^{-1-\alpha} ds + \int_{t_{n,r-1}}^{t_{1,r}} (1+s^{\alpha-1}) (t_{1,r} - s)^{-\alpha} ds,
		\end{split}
	\end{equation}
	which removes the non-integrable singularity of the integrand. Further, by estimating $1+s^{\alpha-1}\leq 2 s^{\alpha-1}$ and evaluating integrals, we obtain
	\begin{equation}
		\begin{split}
			(\Delta t)^{-1}&\sum_{j=0}^{r-1} \int_{t_{1,j}}^{t_{1,j+1}} (1+s^{\alpha-1})(s-t_{1,j})(t_{1,j+1}-s)(t_{1,r}-s)^{-1-\alpha} ds \\
			&\leq 2(\Delta t) (\Delta T)^{\alpha-1} \frac{1}{\alpha} \left((\Delta t)^{-\alpha} - (r \Delta t)^{-\alpha}\right) + 2(\Delta T + (r-1) \Delta t)^{\alpha-1} \frac{1}{1-\alpha} (\Delta t)^{1-\alpha} \\
			&\leq \frac{2 M^{\alpha-1}}{\alpha} + \frac{2}{1-\alpha} (M+r-1)^{\alpha-1} \leq \frac{2 M^{\alpha-1}}{\alpha(1-\alpha)}. 
		\end{split}
	\end{equation}
	This concludes the case $n=0,1$.
	
	\medskip\noindent\textbf{The case $\mathbf{n>1}$.} Now, we go back to \eqref{eqn:TruncationGeneralError} and first estimate the history. Splitting the term $i=0$, since there is a singularity, using the estimate $s^{\alpha-1} \geq 1$, and summing all over the remaining subintervals, we obtain the following
	\begin{equation}
		\begin{split}
			(\Delta T)^{-1}&\sum_{i=0}^{n-1} \int_{T_i}^{T_{i+1}} (1+s^{\alpha-1})(s-T_i)(T_{i+1}-s)(t_{n,r}-s)^{-1-\alpha} ds \\
			&\leq 2(\Delta T)^{-1} \int_0^{\Delta T} s^\alpha (\Delta T -s)(t_{n,r}-s)^{-1-\alpha} ds + \Delta T\int_{\Delta T}^{T_n}(1+s^{\alpha-1})(t_{n,r}-s)^{-1-\alpha}ds,
		\end{split}
	\end{equation}
	which further can be estimated by evaluating the integral
	\begin{equation}
		\begin{split}
			(\Delta T)^{-1}&\sum_{i=0}^{n-1} \int_{T_i}^{T_{i+1}} (1+s^{\alpha-1})(s-T_i)(T_{i+1}-s)(t_{n,r}-s)^{-1-\alpha} ds \\
			&\leq \frac{2}{1+\alpha} M^{\alpha+1}(\Delta t)^{\alpha+1} t_{n-1,r}^{-\alpha-1} + \Delta T\int_{\Delta T}^{T_n}(1+s^{\alpha-1})(t_{n,r}-s)^{-1-\alpha}ds = I_1 + I_2.
		\end{split}
	\end{equation}
	The estimate for $I_1$ is straightforward since $\Delta t < 1$ and $t_{n-1,r}^{-\alpha-1} \leq T_{n-1}^{-\alpha-1} \leq T^{\alpha-1}_{n-1}$, where the last inequality holds for $n\geq 2$. That is,  
	\begin{equation}
		I_1 \leq \frac{2}{1+\alpha} M^{\alpha+1}(\Delta t)^{1-\alpha} T_{n-1}^{\alpha-1}. 
	\end{equation}
	In turn, $I_2$ can be estimated by splitting the integrand and substituting $s=t_{n,r}w$ in the second term
	\begin{equation}
		\begin{split}
			I_2 &\leq \frac{\Delta T}{\alpha} \left(t_{n-1,r}^{-\alpha} - t_{n,r}^{-\alpha}\right) + (\Delta T) t_{n,r}^{-1}\int_{\frac{1}{n+r/M}}^{1-\frac{1}{1+nM/r}} w^{\alpha-1} (1-w)^{-\alpha-1} dw \\
			&\leq \frac{M^{\alpha}}{\alpha}(\Delta t) t_{n-1,r}^{-\alpha} + (\Delta T) t_{n,r}^{-1} \left(\int_{\frac{1}{n+r/M}}^{\frac{1}{2}} w^{\alpha-1} (1-w)^{-\alpha-1} dw+\int_{\frac{1}{2}}^{1-\frac{1}{1+nM/r}} w^{\alpha-1} (1-w)^{-\alpha-1} dw\right).
		\end{split}
	\end{equation}
	The splitting in $w = 1/2$ helps to separate two singularities at $w = 0$ and $w = 1$ that are approached asymptotically when $n\rightarrow\infty$. Next, by elementary estimates $(1-w)^{-\alpha-1}\leq 2^{\alpha+1}$ in the first and $w^{\alpha-1} \leq 2^{\alpha-1}$ in the second integral, we obtain
	\begin{equation}
		\begin{split}
			I_2 &\leq \frac{M^{\alpha}}{\alpha}(\Delta t) t_{n-1,r}^{-\alpha} + (\Delta T) t_{n,r}^{-1} \left(\frac{2^{\alpha+1}}{\alpha} \left(2^{-\alpha} - \left(n+\frac{r}{M}\right)^{-\alpha}\right)  +\frac{2^{\alpha-1}}{\alpha}\left(\left(1+\frac{nM}{r}\right)^\alpha - 2^\alpha\right)\right).
		\end{split}
	\end{equation}
	Therefore, by neglecting the negative terms, 
	\begin{equation}
		\begin{split}
			I_2 &\leq \frac{C}{\alpha}\left[M^{\alpha}(\Delta t) t_{n-1,r}^{-\alpha} + (\Delta T) t_{n,r}^{-1} \left(1 + \left(1+\frac{nM}{r}\right)^\alpha\right)\right].
		\end{split}
	\end{equation}
	Now, observe that
	\begin{equation}
		\begin{split}
			&(\Delta t)t_{n-1,r}^{-\alpha}
			\leq (\Delta t)^{1-\alpha} M^{-\alpha}, \\ 
			&(\Delta T) t_{n,r}^{-1} = \frac{\Delta T}{n\Delta T + r \Delta t} \leq n^{-1} = (\Delta t)^{1-\alpha} (n\Delta t)^{\alpha-1} n^{-\alpha} \leq (\Delta t)^{1-\alpha} T_n^{\alpha-1} M^{\alpha-1}, \\
			&(\Delta T) t_{n,r}^{-1} \left(1+\frac{nM}{r}\right)^\alpha \leq n^{\alpha-1} \left(n^{-1}+\frac{M}{r}\right)^\alpha \leq (\Delta t)^{1-\alpha} T_n^{\alpha-1} M^{1-\alpha} (1+M)^\alpha.
		\end{split}
	\end{equation}
	Therefore, we have a bound for $I_2$ as follows
	\begin{equation}
		I_2 \leq \frac{C}{\alpha}(\Delta t)^{1-\alpha} \left[1+T_n^{\alpha-1}\left(M^{\alpha-1} + M^{1-\alpha} (1+M)^\alpha\right) \right] \leq \frac{C}{\alpha}(\Delta t)^{1-\alpha}T_n^{\alpha-1}\left[T^{1-\alpha}+M^{\alpha-1} + M^{1-\alpha} (1+M)^\alpha\right].
	\end{equation}
	Finally, gathering the $I_1$ and $I_2$, the history term can be estimated 
	\begin{equation}
		\begin{split}
			(\Delta T)^{-1}&\sum_{i=0}^{n-1} \int_{T_i}^{T_{i+1}} (1+s^{\alpha-1})(s-T_i)(T_{i+1}-s)(t_{n,r}-s)^{-1-\alpha} ds \\
			&\leq C (\Delta t)^{1-\alpha}T_{n-1}^{\alpha-1}\left[M^{\alpha+1} + \frac{1}{\alpha}\left(T^{1-\alpha}+M^{\alpha-1} + M^{1-\alpha} (1+M)^\alpha\right)\right] \leq C (\Delta t)^{1-\alpha}T_{n-1}^{\alpha-1} \frac{M^{\alpha+1}}{\alpha}.
		\end{split}
	\end{equation}
	Now, the error in the latest interval in \eqref{eqn:TruncationGeneralError} can be estimated by, again, splitting the last integral and appropriately bounding the node-related terms $(s-t_{n,j})(t_{n,j+1}-s)$
	\begin{equation}
		\begin{split}
			(\Delta t)^{-1}&\sum_{j=0}^{r-1} \int_{t_{n,j}}^{t_{n,j+1}} (1+s^{\alpha-1})(s-t_{n,j})(t_{n,j+1}-s)(t_{n,r}-s)^{-1-\alpha} ds \\
			&\leq (\Delta t) \sum_{j=0}^{r-2} \int_{t_{n,j}}^{t_{n,j+1}}(1+s^{\alpha-1})(t_{n,r}-s)^{-1-\alpha} ds + \int_{t_{n,r-1}}^{t_{n,r}} (1+s^{\alpha-1})(t_{n,r}-s)^{-\alpha}ds.
		\end{split}
	\end{equation}
	We can now bound $(1+s^{\alpha-1})^{\alpha-1}$ from the above and evaluate the sum and integrals
	\begin{equation}
		\begin{split}
			(\Delta t)^{-1}&\sum_{j=0}^{r-1} \int_{t_{n,j}}^{t_{n,j+1}} (1+s^{\alpha-1})(s-t_{n,j})(t_{n,j+1}-s)(t_{n,r}-s)^{-1-\alpha} ds \\
			&\leq (1+T_n^{\alpha-1}) \left[\frac{\Delta t}{\alpha} \left((\Delta t)^{-\alpha} - (r\Delta t)^{-\alpha}\right) + \frac{1}{1-\alpha} (\Delta t)^{1-\alpha}\right] \leq \frac{C}{\alpha(1-\alpha)} (\Delta t)^{1-\alpha} T_n^{\alpha-1}.
		\end{split}
	\end{equation}
	This last estimate completes the proof. 
\end{proof}

\begin{rem}\label{rem:Regularity}
	From the above proof it is easy to see that if the function $y=y(t)$ is $C^2[0,T]$ we would obtain a discretisation error of order $(\Delta t)^{2-\alpha}$ which is the best possible. However, such a smoothness degree of solutions to the subdiffusion equation can only happen when the coefficients and the data are appropriately compatible (see \cite{Sak11, Sty16}).  
\end{rem}

\begin{rem}
	From Proposition \ref{prop:L1Hybrid} we see the typical feature of the L1 scheme: the pointwise error keeps a constant order (here, under our regularity assumption equal to $1-\alpha$) away from the origin. However, the error deteriorates close to $t = 0$, as was previously observed in the standard L1 discretisation in \cite{stynes2017error}. In that work, the authors obtained a better order of convergence under stronger regularity assumptions. 
\end{rem}

\subsection{Convergence of the parareal algorithm}

Before we proceed to the main result concerning convergence, we state the important lemma of discrete analysis. This can be thought of as a more complex version of the standard discrete Gronwall inequality. 

\begin{lem}\label{lem:Gronwall}
	Let $\left\{E^k_n\right\}_{n,k \in \mathbb{N}}$ be a positive sequence. If
	\begin{equation}
		E^{k+1}_{n+1} \leq a + b E^k_n + c E^{k+1}_{n}, \quad E^k_0 = 0, 
	\end{equation}
	for $a,b,c > 0$, then
	\begin{equation}\label{eqn:LemmaInequality}
		E^k_n \leq a \sum_{j=0}^{k-1} \sum_{i=0}^{n-j-1} \binom{i+j}{j} c^i b^j + E^0_n b^k \sum_{j=k}^{n-1} \binom{j-1}{k-1} c^{j-k}.
	\end{equation}
\end{lem}
\begin{proof}
	The proof proceeds in three steps:
	\begin{enumerate}
		\item Formulate auxiliary \emph{equation}.
		\item Solve the homogeneous equation.
		\item Solve the non-homogeneous problem.
	\end{enumerate}
	\medskip\noindent\textbf{Ad. 1.} First, define $\{f^k_n\}_{n,k\in\mathbb{N}}$ to be the solution of the following equation
	\begin{equation}
		f^{k+1}_{n+1} = a + b f^k_n + c f^{k+1}_{n}, \quad f^k_0 = 0, \quad f^0_n = E^0_n.
	\end{equation}
	We will show by induction that $E^k_n \leq f^k_n$. This shows that it is sufficient only to solve the equation. By the definition of the initial values $f_0^k$ and $f^0_n$ we see that the first step in the induction is satisfied. Now, assume that $E^i_j \leq f^i_j$ for all $0\leq i \leq k+1$ and $0\leq n \leq n$. By this assumption and the definition of $E^{k+1}_{n+1}$, we have
	\begin{equation}
		E^{k+1}_{n+1} \leq a + b E^k_n + c E^{k+1}_n \leq a + b f^{k}_n + c f^{k+1}_n = f^{k+1}_{n+1}, 
	\end{equation}
	and the induction is complete. 
	
	\medskip\noindent\textbf{Ad. 2.} We proceed to solving the main equation. By linearity, we can first consider the homogeneous version
	\begin{equation}
		f^{k+1}_{n+1} = b f^k_n + c f^{k+1}_{n}, \quad f^k_0 = 0, \quad f^0_n = E^0_n.
	\end{equation}
	We will that
	\begin{equation}
		f^k_n = E^0_n b^k \sum_{j=k}^{n-1} \binom{j-1}{k-1} c^{j-k}.
	\end{equation}
	This formula can be found by several techniques: "guessing" based on numerical computations, using generating functions, or combinatorial counting argument. To verify the above, we have to check whether the ansatz satisfies the equation. The "initial condition" is 
	\begin{equation}
		f^0_n = E^0_n \binom{-1}{-1} = E^0_n,
	\end{equation}
	since all other binomial coefficients vanish. On the other hand, the right-hand side of the homogeneous equation is
	\begin{equation}
		\begin{split}
			b f^k_n + c f^{k+1}_{n} 
			&= E^0_n \left( b^{k+1} \sum_{j=k}^{n-1} \binom{j-1}{k-1} c^{j-k} + b^{k+1} \sum_{j=k+1}^{n-1} \binom{j-1}{k} c^{j-k-1+1}\right) \\
			&= E^0_n b^{k+1} \left( \sum_{j=k}^{n-1} \binom{j-1}{k-1} c^{j-k} + \sum_{j=k+1}^{n-1} \binom{j-1}{k} c^{j-k}\right).
		\end{split}
	\end{equation} 
	Now, we separate the first term in the first sum and merge
	\begin{equation}
		b f^k_n + c f^{k+1}_{n} = E^0_n b^{k+1} \left(1+ \sum_{j=k+1}^{n-1} \left[\binom{j-1}{k-1} + \binom{j-1}{k}\right]c^{j-k}\right) = E^0_n b^{k+1} \left(1+ \sum_{j=k+1}^{n-1}\binom{j}{k} c^{j-k}\right),
	\end{equation} 
	where we used the fundamental recurrence identify for the binomial coefficient. Finally, we change the summation index $i = j+1$ and obtain
	\begin{equation}
		b f^k_n + c f^{k+1}_{n} = E^0_n b^{k+1} \left(1+ \sum_{i=k+2}^{n-1}\binom{i-1}{k} c^{i-k-1}\right) = E^0_n b^{k+1} \sum_{i=k+1}^{n-1}\binom{i-1}{k} c^{i-k-1} = f^{k+1}_{n+1},
	\end{equation} 
	which concludes the proof of this part. 
	
	\medskip\noindent\textbf{Ad. 3.} Now, we turn to the nonhomogeneous problem with vanishing initial condition
	\begin{equation}
		f^{k+1}_{n+1} = a + b f^k_n + c f^{k+1}_{n}, \quad f^k_0 = 0, \quad f^0_n = 0.
	\end{equation}
	We claim that
	\begin{equation}
		f^k_n = a \sum_{j=0}^{k-1} \sum_{i=0}^{n-j-1} \binom{i+j}{j} c^i b^j.
	\end{equation}
	First, we see that by the convention in which the upper index of the summation is smaller than the lower, we have $f^0_n = 0$ and $f^k_0 = 0$. Moreover, the ansatz satisfies the equation as can be seen by changing the summation indices
	\begin{equation}
		\begin{split}
			a + b f^k_n + c f^{k+1}_{n} 
			&= a\left(1+\sum_{j=0}^{k-1} \sum_{i=0}^{n-j} \binom{i+j}{j} c^i b^{j+1} + \sum_{j=0}^{k} \sum_{i=0}^{n-j} \binom{i+j}{j} c^{i+1} b^j\right) \\
			&= a\left(1+\sum_{j=1}^{k} \sum_{i=0}^{n-j} \binom{i+j-1}{j-1} c^i b^{j} + \sum_{j=0}^{k} \sum_{i=1}^{n-j} \binom{i+j-1}{j} c^{i} b^j\right).
		\end{split}
	\end{equation}
	Then, by noticing that since $\binom{p}{q} = 0$ if $q > p$, all the $(i,j)$-th terms of the above sums vanish apart the one with $i = j = 0$ with a value $1$, we can artificially add and subtract it
	\begin{equation}
		\begin{split}
			a + b f^k_n + c f^{k+1}_{n} 
			&= a\left(1-1+\sum_{j=0}^{k} \sum_{i=0}^{n-j} \binom{i+j-1}{j-1} c^i b^{j} + \sum_{j=0}^{k} \sum_{i=0}^{n-j} \binom{i+j-1}{j} c^{i} b^j\right) \\
			&= a\sum_{j=0}^{k} \sum_{i=0}^{n-j} \left[\binom{i+j-1}{j-1} + \binom{i+j-1}{j}  \right] c^i b^{j} = a\sum_{j=0}^{k} \sum_{i=0}^{n-j} \binom{i+j}{j} c^i b^{j} = f^{k+1}_{n+1},
		\end{split}
	\end{equation}
	which completes the proof of the lemma.
\end{proof}

\begin{rem}
	Notice that for all $k \geq n$ we have 
	\begin{equation}
		E^k_n \leq a \sum_{j=0}^{n-1} \sum_{i=0}^{n-j-1} \binom{i+j}{j} c^i b^j, \quad k \geq n,
	\end{equation}
	that is, the bound saturates at a $k$-independent value. This follows from the fact that in \eqref{eqn:LemmaInequality} the terms $i \geq n$ in the $a$-sum vanish. Similarly, $\binom{j-1}{k-1} = 0$ for $k\geq j$ in the $E_n^0$-sum.  
\end{rem}

Now, we are able to proceed to the main result concerning convergence of the parareal algorithm.

\begin{thm}\label{thm:Convergence}
	Let $u=u(x,t)$ be the solution of \eqref{eqn:MainPDEWeak} with regularity \eqref{eqn:DiffusionRegularity} in time and $H^1_0(\Omega)$ in space. $U^k_n$ is the parareal approximation at $t=T_n$ with $k$ iterations obtained by \eqref{eqn:Parareal} and $N$ degrees of freedom of the Galerkin approximation. Assume that the parareal method uses coarse $\mathcal{G}$ and fine $\mathcal{F}$ integrators satisfying the following Lipschitz condition 
	\begin{equation}
		\begin{split}
			\|\mathcal{G}_n(V_{0:n}) - \mathcal{G}_n(W_{0:n})\| &\leq C_\mathcal{G}(\Delta(T)) \max_{0\leq j\leq n} \|V_j - W_j\|, \\ 
			\|\mathcal{F}_n(V_{0:n}) - \mathcal{F}_n(W_{0:n})\| &\leq C_\mathcal{F}(\Delta(T),\Delta t) \max_{0\leq j\leq n} \|V_j - W_j\|,
		\end{split}
	\end{equation}
	for arbitrary sequences of $L^2$ functions $\left\{V_j\right\}_j$ and $\left\{W_j\right\}_j$ with Lipschitz constants $C_\mathcal{G}, C_\mathcal{F} > 1$. By $\mathcal{G}_n(U_{0:n, N})$ and $\mathcal{F}_n(U_{0:n, N})$ we denote the sequential approximations obtained by these propagators. Then, 
	\begin{equation}\label{eqn:PararealError}
		\begin{split}
			\|u(T_n) - U^k_{n,N}\| \leq \frac{a}{c-1} \left(c(b+c)^{\min\{k,n\}-1} - (b+1)^{\min\{k,n\}-1}\right) &\max_{0\leq j\leq N_t}\|u(T_{j+1}) - \mathcal{F}_j(U_{0:j, N})\| \\
			+b^k c^{n-k-1} \binom{n-1}{k} &\max_{0\leq j\leq N_t}\|u(T_{j+1}) - \mathcal{G}_j(U_{0:j, N})\|,
		\end{split}
	\end{equation}
	where
	\begin{equation}
		a := (1+ C_\mathcal{F}(\Delta(T),\Delta t)), \quad b := C_\mathcal{F}(\Delta(T),\Delta t)+C_\mathcal{G}(\Delta(T)), \quad c := C_\mathcal{G}(\Delta(T)).
	\end{equation}
\end{thm}
\begin{proof}
	Denote the parareal error by $e^{k}_{n} := u(T_{n}) - U_{n,N}^{k}$ and by \eqref{eqn:Parareal} write
	\begin{equation}
		e^{k+1}_{n+1} = u(T_{n+1}) - \mathcal{F}_n(U_{0:n, N}) + \mathcal{F}_n(U_{0:n, N}) - \mathcal{F}_n(U_{0:n, N}^{k}) + \mathcal{G}_n(U_{0:n, N}^{k+1}) - \mathcal{G}_n(U_{0:n, N}^k),
	\end{equation}
	where $U_{n,N}$ is the sequential solution of the fine scheme \eqref{eqn:NumericalSchemeFine}. Take the inner product of the above with $e^{k+1}_{n+1}$ and use the Cauchy-Schwarz inequality
	\begin{equation}
		\|e^{k+1}_{n+1}\|^2 \leq \left(\|u(T_{n+1}) - \mathcal{F}_n(U_{0:n, N})\| + \|\mathcal{F}_n(U_{0:n, N}) - \mathcal{F}_n(U_{0:n, N}^{k})\| + \|\mathcal{G}_n(U_{0:n, N}^{k+1}) - \mathcal{G}_n(U_{0:n, N}^k)\| \right)\|e^{k+1}_{n+1}\|.
	\end{equation}
	The first term in parentheses is the fine integrator error because $\mathcal{F}_n(U_{0:n, N}) = U_{n+1,N}$. Using the Lipschitz condition for the integrators we further obtain
	\begin{equation}
		\|e^{k+1}_{n+1}\| \leq \|u(T_{n+1}) - \mathcal{F}_n(U_{0:n, N})\| + C_\mathcal{F}(\Delta(T),\Delta t) \max_{0\leq j\leq n} \|U_{j,N} - U_{j,N}^k\| + C_\mathcal{G}(\Delta(T)) \max_{0\leq j\leq n} \|U_{j,N}^{k+1} - U_{j,N}^{k}\|.
	\end{equation}
	Now, we have
	\begin{equation}
		\|U_{j,N} - U_{j,N}^k\| \leq \|U_{j,N} - u(T_j)\| + \|u(T_j) - U_{j,N}^k\| = \|U_{j,N} - u(T_j)\| + \|e^{k}_j\|,
	\end{equation}
	and similarly
	\begin{equation}
		\|U_{j,N}^{k+1} - U_{j,N}^{k}\| \leq \|U_{j,N}^{k+1} - u(T_j)\| + \|u(T_j) - U_{j,N}^{k}\| = \|e^{k+1}_j\| + \|e^k_j\|. 
	\end{equation}
	Define the running maximal error
	\begin{equation}
		E^{k}_n := \max_{0\leq j\leq n} \|e^k_j\|,
	\end{equation}
	which implies the main recurrence inequality
	\begin{equation}
		\begin{split}
			E^{k+1}_{n+1} &\leq (1+ C_\mathcal{F}(\Delta(T),\Delta t))\max_{0\leq j\leq N_t}\|u(T_{j+1}) - \mathcal{F}_j(U_{0:j, N})\| \\
			&+ \left(C_\mathcal{F}(\Delta(T),\Delta t)+C_\mathcal{G}(\Delta(T)) \right) E^{k}_n + C_\mathcal{G}(\Delta(T))  E^{k+1}_n =: a + b E^k_n + c E^{k+1}_n.
		\end{split}
	\end{equation}
	The above is precisely in the form required by Lemma \ref{lem:Gronwall} which lets us solve the recurrence
	\begin{equation}
		E^{k}_n \leq a \sum_{j=0}^{k-1} \sum_{i=0}^{n-j-1} \binom{i+j}{j} c^i b^j + E^0_n b^k \sum_{j=k}^{n-1} \binom{j-1}{k-1} c^{j-k}.
	\end{equation}
	The sums can be estimated as follows. Since the binomial coefficient is an increasing function of the upper variable, we have $\binom{i+j}{j} \leq \binom{n-1}{j}$, and hence
	\begin{equation}
		\sum_{j=0}^{k-1} \sum_{i=0}^{n-j-1} \binom{i+j}{j} c^i b^j \leq \sum_{j=0}^{k-1} \binom{n-1}{j} b^j \sum_{i=0}^{n-j-1} c^i \leq \sum_{j=0}^{n-1} \binom{n-1}{j} b^j \frac{c^{n-j}-1}{c-1},
	\end{equation}
	where we have estimated $k \leq n$ and summed the geometric series. Furthermore, by estimating again and summing the binomial series we obtain a simple estimate that illustrates the saturation point for $k\geq n$ and is bounded for $c\rightarrow 1^+$
	\begin{equation}\label{eqn:DoubleSumEst}
		\sum_{j=0}^{k-1} \sum_{i=0}^{n-j-1} \binom{i+j}{j} c^i b^j \leq \frac{1}{c-1}\sum_{j=0}^{n-1} \binom{n-1}{j} b^j \frac{c^{n-j} - 1}{c-1} = \frac{1}{c-1} \left(c(b+c)^{n-1} - (1+b)^{n-1}\right).
	\end{equation}
	In the next sum, we simply estimate $c^{j-k} \leq c^{n-k-1}$ and use the well-known \textit{hockey stick identity} for the binomial $\sum_{i=p}^q \binom{i}{p} = \binom{q+1}{p+1}$
	\begin{equation}\label{eqn:SingleSumEst}
		\sum_{j=k}^{n-1} \binom{j-1}{k-1} c^{j-k} \leq c^{n-k-1} \sum_{j=k}^{n-1} \binom{j-1}{k-1} = c^{n-k-1} \sum_{i=k-1}^{n-2} \binom{i}{k-1} = c^{n-k-1} \binom{n-1}{k}.
	\end{equation}
	Hence, 
	\begin{equation}
		E^{k}_n \leq \frac{ac}{c-1} (b+c)^{n-1} + E^0_n b^k c^{n-k-1} \binom{n-1}{k}.
	\end{equation}
	By unraveling the definitions of constants $a$, $b$, $c$, and recalling that $E^0_n = \max_{0\leq j\leq n} \|u(T_{n+1}) - \mathcal{G}_n(U_{0:n,N})\|$ is the error of the sequential coarse propagator, we conclude the proof. 
\end{proof}

From the main error estimate \eqref{eqn:PararealError} we see that it decomposes into two terms: one involving the coarse and the other the fine integrator errors. When $k\rightarrow n$ the former vanishes leaving only the latter present, confirming our discussion in \eqref{eqn:PararealConvergence} that the parareal algorithm converges to the solution of the fine integrator. Notice that the constants multiplying the propagator errors grow exponentially with $n$. However, in our analysis, we treat $n$ as a fixed number and consider only the convergence of the parareal iterations. Also, it is usual in the a priori analysis that the constants are very large. In practice, adaptive algorithms are implemented that terminate iterations as soon as the a posteriori error estimate is satisfied. Usually, this happens after just a few iterations (as observed in \cite{gander2007analysis}). 

An illustration of the bounds found for the sums that appear that involve binomials as in \eqref{eqn:SingleSumEst} and \eqref{eqn:DoubleSumEst} is shown in Fig. \ref{fig:Sums}. As can be seen, the estimate for the single sum \eqref{eqn:SingleSumEst} is very accurate and vanishes as soon as $k = n$.

\begin{figure}
	\centering
	\includegraphics{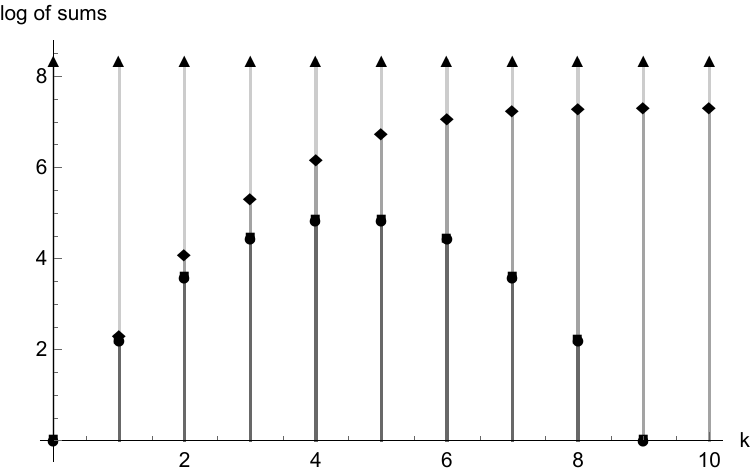}
	\caption{Logarithm of sums involving binomials and their bounds. $\bullet$ is the double sum in \eqref{eqn:DoubleSumEst} and $\blacksquare$ is its estimate. $\blacklozenge$ is the single sum in \eqref{eqn:SingleSumEst} and $\blacktriangle$ is its estimate. Here, $n=10$, $b = 1.1$, $c = 1.001$. }
	\label{fig:Sums}
\end{figure}

We have proved Theorem \ref{thm:Convergence} under the general assumption that both integrators are Lipschitz. Of course, thanks to Proposition \ref{prop:Lipschitz} it is satisfied by $\mathcal{G}$ and $\mathcal{F}$ introduced in Section \ref{sec:Parareal}. For these, more can be said about the error in solving our main problem \eqref{eqn:MainPDE}. 

\begin{cor}
	Let the assumptions of Theorem \ref{thm:Convergence} be satisfied. Assume that the coarse and fine integrators are chosen to be \eqref{eqn:CoarsePropagator} and \eqref{eqn:FinePropagator}, respectively. If $u$ is additionally $H^m(\Omega) \cap H^1_0(\Omega)$ with $m > 1$ in space and the Galerkin scheme is Legendre spectral, then provided that $\|\nabla u\|_\infty < \infty$, we have
	\begin{equation}
		\begin{split}
			&\|u(T_n) - U^k_{n,N}\| \\
			&\leq C\left[\frac{a}{c-1} \left(c(b+c)^{\min\{k,n\}-1} - (b+1)^{\min\{k,n\}-1}\right)\left(N^{-m} + 
			M^{1+\alpha} (\Delta t)^{\frac{\alpha}{2}} 
			\begin{cases}
				1, & \alpha \neq \frac{1}{2}, \\
				\left(\ln (\Delta t)^{-1}\right)^{\frac{1}{2}}, & \alpha = \frac{1}{2}, \\
			\end{cases}\right) \right. \\
			&\left.+ b^k c^{n-k-1} \binom{n-1}{k} \left(N^{-m} + 
			(\Delta T)^{\frac{\alpha}{2}} 
			\begin{cases}
				1, & \alpha \neq \frac{1}{2}, \\
				\left(\ln (\Delta T)^{-1}\right)^{\frac{1}{2}}, & \alpha = \frac{1}{2}, \\
			\end{cases}\right)\right],
		\end{split}
	\end{equation}
	where $C=C(\alpha, u) > 0$ and the constants $a$, $b$, and $c$ are defined according to the Lipschitz constants in Proposition \ref{prop:Lipschitz}. 
\end{cor}
\begin{proof}
	In \cite{plociniczak2023linear} it has been shown that the error in space of the Galerkin spectral discretization of the quasilinear subdiffusion equation \eqref{eqn:MainPDE} is proportional to $N^{-m}$. Thus, we have to focus only on the time discretisation. Exactly in the same way as in the proof of Theorem 3 from \cite{plociniczak2023linear}, using the assumption of the boundedness of the gradient, we can show that (formula (82) in \cite{plociniczak2023linear})
	\begin{equation}
		\begin{split}
			&\delta^\alpha \|u(T_{n+1}) - \mathcal{F}_n(U_{n,0:M,N}, U_{0:n,N})\|^2 \\ 
			&\leq C\left(\|u(T_{n}) - \mathcal{F}_{n-1}(U_{n-1,0:M, N}, U_{0:n-1})\|^2 + N^{-2m} + \|\partial^\alpha u(T_n) - \delta^\alpha u(T_n)\|^2\right),
		\end{split}
	\end{equation}
	and analogously for the coarse integrator $\mathcal{G}$. The above is a standard error estimate that is found in many numerical contexts. It can be resolved by application of the corresponding discrete fractional Gronwall inequality which implies (see \cite{del2025numerical}, Lemma A1)
	\begin{equation}\label{eqn:FinePropagatorError}
		\|u(T_{n+1}) - \mathcal{F}_n(U_{n,0:M,N}, U_{0:n,N})\|^2 \leq C \left(N^{-2m} + \frac{(\Delta t)^\alpha}{\Gamma(\alpha)} \sum_{i=0}^{n-1} (n-i)^{\alpha-1}  \|\partial^\alpha u(T_{i+1}) - \delta^\alpha u(T_{i+1})\|^2\right),
	\end{equation}
	and we have to estimate the sum corresponding to the discrete fractional integral. By Proposition \ref{prop:L1Hybrid} we know the bound for the discretisation error of the Caputo derivative and, hence,
	\begin{equation}
		\begin{split}
			\frac{(\Delta t)^\alpha}{\Gamma(\alpha)} &\sum_{i=0}^{n-1} (n-i)^{\alpha-1}  \|\partial^\alpha u(T_{i+1}) - \delta^\alpha u(T_{i+1})\|^2 \\
			&\leq \frac{1}{\Gamma(\alpha)} (\Delta t)^{\alpha}\left(\left(\frac{M^\alpha}{\Gamma(1-\alpha)}\right)^2 n^{\alpha-1} + M^{4\alpha}  \sum_{i=0}^{n-1} (n-i)^{\alpha-1} (i+1)^{2\alpha-2} \right).
		\end{split}
	\end{equation}
	The appearing sum can be bounded with respect to the value of $\alpha$ thanks to the results from \cite{lopez2025convolution}, formula (3.20)
	\begin{equation}
		\begin{split}
			\frac{(\Delta t)^\alpha}{\Gamma(\alpha)} &\sum_{i=0}^{n-1} (n-i)^{\alpha-1}  \|\partial^\alpha u(T_{i+1}) - \delta^\alpha u(T_{i+1})\|^2 \\
			&\leq \frac{M^{4\alpha}}{\Gamma(\alpha)} (\Delta t)^{\alpha}\left(
			\left(\frac{1}{\Gamma(1-\alpha)}\right)^2 n^{\alpha-1} + 
			\begin{cases}
				n^{3\alpha-2}, & \frac{1}{2}<\alpha<1, \\
				n^{-\frac{1}{2}} \ln n, & n = \frac{1}{2}, \\
				n^{\alpha-1}, & 0< \alpha < \frac{1}{2}, \\
			\end{cases}
			\right).
		\end{split}
	\end{equation}
	Now, for $\alpha > 2/3$, we can write $n^{3\alpha-2} = T_n^{3\alpha-2} (\Delta T)^{2-3\alpha} \leq  T^{3\alpha-2} (\Delta T)^{2-3\alpha}$. Similarly, for $1/2 < \alpha < 2/3$ we have $n^{3\alpha-2} \leq 1$, for $\alpha = 1/2$ we have $n^{-1/2} \ln n \leq \ln T_n - \ln (\Delta T) \leq C \ln(\Delta T)^{-1}$, and for $0<\alpha<1$ we see that $n^{\alpha-1} \leq 1$. Therefore,
	\begin{equation}
		\begin{split}
			\frac{(\Delta t)^\alpha}{\Gamma(\alpha)} &\sum_{i=0}^{n-1} (n-i)^{\alpha-1}  \|\partial^\alpha u(T_{i+1}) - \delta^\alpha u(T_{i+1})\|^2 \\
			&\leq C\frac{M^{4\alpha}}{\Gamma(\alpha)} (\Delta t)^{\alpha}\left(
			\left(\frac{1}{\Gamma(1-\alpha)}\right)^2 + 
			\begin{cases}
				(M\Delta t)^{2-2\alpha}, & \frac{2}{3} < \alpha < 1, \\
				1, & \frac{1}{2}<\alpha<\frac{2}{3}, \\
				\ln (\Delta t)^{-1}, & n = \frac{1}{2}, \\
				1, & 0< \alpha < \frac{1}{2}, \\
			\end{cases}
			\right)\\
			&\leq C(\alpha) M^{2+2\alpha} (\Delta t)^\alpha 
			\begin{cases}
				1, & \alpha \neq \frac{1}{2}, \\
				\ln (\Delta t)^{-1}, & n = \frac{1}{2}. \\
			\end{cases}
		\end{split}
	\end{equation}
	Returning to \eqref{eqn:FinePropagatorError}, taking the maximum with respect to $n$, and the square root, we arrive at the final form of the fine integrator error. The proof for the coarse integrator is completely analogous. The proof is thus completed.  
\end{proof}

\begin{rem}
	The error order in time found in the above corollary is not globally optimal. The reason for this is the nonlinear diffusivity in the main equation \eqref{eqn:MainPDE} as observed in \cite{lopez2025convolution}. This forces a different proof technique in the $L^2$ setting that increases the error estimate. However, note that we assume the \textit{minimal} realistic regularity of the solution. In the case of higher smoothness over time, the error estimate would be globally optimal. For example, as noticed in Remark \ref{rem:Regularity}, having a solution that is $C^2[0,T]$ in time would produce the best possible order $2-\alpha > 1$. However, we note that this often assumed regularity is not realistic and occurs very rarely. In our analysis, we decided to work in the natural degree of smoothness.
\end{rem}

\begin{rem}
	In \cite{xu2015parareal}, the authors analyze a parareal algorithm applied to systems of fractional differential equations using the L1 time discretization scheme. Although they establish a stability bound for the parareal iteration under Lipschitz continuity of the propagators, their estimate involves a factor of the form \( 1 + C\Delta t \). However, as demonstrated in Proposition \ref{prop:Lipschitz} above, for fractional problems, this factor should be more accurately expressed as \( 1 + C(\Delta t)^\alpha \), reflecting the nonlocal nature of the fractional derivative.
	
	Furthermore, in Theorem 3.5, the authors derive an error bound involving a term \( c^{k+1} e^{c k} \). A straightforward analysis shows that this quantity is bounded only if \( c < z_0 \), where \( z_0 \approx 0.567143 \) is the unique positive solution to equation \( z + \ln z = 0 \) (a product log), assuming \( T_n = 1 \) for simplicity. However, their assumptions and proof do not guarantee that \( c < z_0 \). Although the authors suggest that \( c \) may be negative in certain parabolic discretizations, their actual proof does not support this claim. In fact, the argument suggests that \( c > 0 \) always holds. Hence, the convergence estimate provided may not be valid in the generality asserted.
\end{rem}

\section{Numerical illustration}
In what follows, we describe the practical implementation of our scheme and present some numerical examples verifying the claimed efficiency of the constructed parareal method.

\subsection{Implementation with the spectral scheme in space}

So far, we have talked mostly about the temporal disretisation to which the parareal method was applied. To facilitate the spatial discretisation and to make computations as fast as possible, we use the Chebyshev--Lobatto spectral collocation method. Here, we briefly describe its basic principles, but the interested reader can find much more details in monographs \cite{canuto2006spectral, trefethen2000spectral}. 

For simplicity, we first describe the method on the reference interval \([-1,1]\) on which the scheme is most easily defined. The Chebyshev--Lobatto (C--L) nodes are
\begin{equation}
	\xi_j = \cos\!\bigg(\frac{j\pi}{N}\bigg),\qquad j=0,\dots,N.
\end{equation}
Then, on a physical 1D interval \([a,b]\) we use the affine map
\begin{equation}
	x = \frac{a+b}{2} + \frac{b-a}{2}\,\xi,
\end{equation}
so that the derivatives are transformed as \(\partial_x = 2(b-a)^{-1}\partial_\xi\). For rectangular multidimensional domains \(\Omega = \prod_{r=1}^d [a_r,b_r]\), we apply the 1D mapping in each coordinate. Now, we approximate \(u(\cdot,t)\) by the degree-$N$ Chebyshev interpolant
\begin{equation}
	u_N(\xi,t) = \sum_{k=0}^N a_k(t) T_k(\xi),
\end{equation}
or equivalently by its values at the C--L nodes (the relation between coefficients $a_k(t)$ and pointwise nodal values, thanks to the properties of Chebyshev polynomials, can be easily obtained via the discrete cosine transform)
\begin{equation}
	u_{j}(t) \approx u_N(\xi_j,t),\qquad j=0,\dots,N.
\end{equation}
In practice, the collocation (pseudospectral) formulation works directly with the vector of values \(\mathbf{u}(t) = [u_0(t),u_1(t),\dots,u_N(t)]^\top\). The key object in the discretisation of the subdiffusion equation is the \textit{differentiation matrix} (see \cite{trefethen2000spectral}, Chapter 6). Define weights
\begin{equation}
	c_j=\begin{cases} 2,& j=0,N,\\ 1,& 1\le j\le N-1. \end{cases}
\end{equation}
The first derivative (spectral) differentiation matrix \(D\in\mathbb{R}^{(N+1)\times(N+1)}\) has entries 
\begin{equation}
	D_{ij} = 
	\begin{cases}
		\displaystyle \frac{c_i}{c_j}\,\frac{(-1)^{i+j}}{x_i-x_j}, & i\neq j,\\[8pt]
		\displaystyle -\sum_{\substack{m=0\\ m\neq i}}^N D_{im}, & i=j.
	\end{cases}
\end{equation}
Equivalently, diagonal entries can be written explicitly as
\begin{equation}
	D_{ii} = 
	\begin{cases}
		\displaystyle \frac{2N^2+1}{6}, & i=0, \\[6pt]
		\displaystyle -\frac{x_i}{2(1-x_i^2)}, & 1\le i\le N-1,\\[6pt]
		\displaystyle -\frac{2N^2+1}{6}, & i=N.
	\end{cases}
\end{equation}
Of course, on a physical interval \([a,b]\) we scale \(D\) by the factor \(\frac{2}{b-a}\), that is, \(D_{[a,b]} = \frac{2}{b-a}\,D\). Furthermore, the second derivative matrix is \(D^{(2)}=D^2\). Finally, acting on the vector of nodal values yields
\begin{equation}
	\mathbf{u}' \approx D\mathbf{u},\qquad \mathbf{u}'' \approx D^{(2)}\mathbf{u}.
\end{equation}
So derivative operators are approximated by matrices. In higher dimensions, for \(d\ge 2\) in a rectangular domain \(\Omega=\prod_{r=1}^d[a_r,b_r]\) we form tensor product C--L grids. For example, in 2D with \(N_x,N_y\) nodes in the \(x,y\) directions the grid nodes are
\begin{equation}
	(x_i,y_j) = \Big( \,x(\xi_i),\, y(\eta_j)\Big),\qquad \xi_i=\cos\frac{i\pi}{N_x},\;\eta_j=\cos\frac{j\pi}{N_y}.
\end{equation}
Let \(D_x\) be the 1D differentiation matrix in the \(x\)-direction (scaled to the physical interval \([a_x,b_x]\)) and \(D_y\) the corresponding \(y\)-matrix. Using the Kronecker product $\otimes$, the 2D discrete derivative matrices acting on the vectorized (lexicographically ordered) solution \(\mathbf{u}\) are
\begin{equation}
	\mathcal{D}_x = D_x \otimes I_{N_y+1}, \qquad	\mathcal{D}_y = I_{N_x+1} \otimes D_y,
\end{equation}
and the constant-coefficient Laplacian reads
\begin{equation}
	\Delta_h \;=\; \mathcal{D}_x^2 + \mathcal{D}_y^2.
\end{equation}
Analogous constructions are held to higher dimensions by additional Kronecker products.

Since, in general, our equation \eqref{eqn:MainPDE} is quasilinear, we have to speciffically construct the nonlinear diffusivity operator. The diffusion term is \(\nabla \cdot (D(x,t,u)\nabla u)\). In collocation form, we store the nodal values of the diffusion coefficient
\begin{equation}
	D_{p} = D(x_p,t,u_p),\qquad p \text{ indexes grid nodes}.
\end{equation}
Let \(\mathbf{D}(\mathbf{u},t)=\operatorname{diag}(D(x_p,t,u_p))\) be the diagonal matrix of nodal diffusion values (in the same ordering as the unknowns). The discrete approximation of \(\nabla \cdot (D(x,t,u)\nabla u)\) is then formed by applying the discrete divergence to the vector of discrete fluxes
\begin{equation}
	\mathcal{N}(\mathbf{u},t) \mathbf{u} := \sum_{r=1}^d \mathcal{D}_r\!\big(\mathbf{D}(\mathbf{u},t)\,\mathcal{D}_r\mathbf{u}\big),
\end{equation}
where \(\mathcal{D}_r\) is the discrete derivative in coordinate \(r\) (Kronecker product form). This yields a (generally) nonlinear algebraic operator \(\mathcal{N}(\mathbf{u},t)\) because \(\mathbf{D}\) depends on \(\mathbf{u}\). Homogeneous Dirichlet BCs \(u=0\) on \(\partial\Omega\) are enforced at the C--L boundary nodes by separating unknowns into interior and boundary nodes:
\begin{equation}
	\mathbf{u} = \begin{bmatrix}\mathbf{u}_I\\ \mathbf{u}_B\end{bmatrix}, \quad \mathbf{u}_B \equiv \mathbf{0}.
\end{equation}
Time stepping is then performed for interior unknowns only, and we refrain from writing subscript $I$ for brevity.

After spatial discretisation and the application of the semi-implicit L1 scheme in time (as in \eqref{eqn:NumericalSchemeSequential} and \eqref{eqn:NumericalSchemeFine}), we obtain a fully-discrete \textit{linear} system for the vector of interior unknowns \(\mathbf{u}\)
\begin{equation}\label{eqn:NumericalSchemeMainSystem}
	\delta_t^\alpha \mathbf{u}_n = \mathcal{N}\left(\mathbf{u}_{n-1} ,t_{n-1}\right) \mathbf{u}_n + \mathbf{f}\left(\mathbf{u}_{n-1} , t_{n-1}\right),
\end{equation}
where \(\mathbf{f}\) is the collocated source vector (restricted to interior nodes). This system can then be solved by a dedicated linear algebra package. Efficient solvers or preconditioners that exploit the Kronecker structure can be used.

\subsection{Examples}
We can now illustrate the efficiency of our algorithm with a concrete numerical example. The computer code has been implemented in the Julia programming language with the following performance improvements:
\begin{itemize}
	\item Use \texttt{@views} to avoid array copies when slicing;
	\item Preallocate all large work arrays outside of the time-stepping loops;
	\item Exploit \texttt{Threads.@threads} for parallelising the fine propagator over coarse intervals in each parareal iteration;
	\item Reuse differentiation matrices and quadrature weights for all time steps;
	\item Avoid repeated allocation of history term arrays in the L1 scheme by maintaining rolling buffers.
\end{itemize}
All our computations have been run on a 12th-Gen Intel(R) Core(TM) i7-1265U processor with 12 threads in parallel. 

As a realistic model problem we choose \eqref{eqn:MainPDE} with
\begin{equation}
	\Omega = [0,1], \quad D(x,t,u) = 1 + u, \quad f(x,t,u) = \sin(\pi x) e^{-t}, \quad u_0(x) = x^4(1-x)^4, \quad \alpha = 0.5,
\end{equation}
which provides a sufficient amount of nonlinearity in the problem. We have also tested our code on different examples, obtaining quantitatively the same result. 

The first test is how the parareal iterations converge. All of our simulations (for different $N_t$, $N$, $\alpha$, and $M$) show the same conclusion illustrated in Fig. \ref{fig:errorIterations}: the error at $t = T$ between the parareal iterations and the (purely sequential) fine solution. As can be seen, our computations confirm that parareal iterations converge after only a few steps. This observation is robust: we have seen it in a series of other simulations and it has also been reported in the literature \cite{gander2025time}. Note, however, that the error does not converge to the machine epsilon as the pure theory predicts. This happens for many potential reasons. In our implementation we have multiple sources of non-exactness (nonlinear diffusivity, truncation from the L1 scheme, semi-implicit linearisation, rounding in linear solves, finite precision arithmetic, conditioning of the spatial operator, slight differences in operation order). Those errors accumulate and are amplified by the conditioning of the linear systems. Therefore, although we do not observe the convergence to the smallest possible value, the error still saturates at a decently small level. Playing with parameters can reduce this error further. We can thus confirm the validity of our Theorem \ref{thm:Convergence} with respect to convergence and stating that the error of the parareal method is bounded by the error of the fine integrator. 

\begin{figure}
	\centering
	\includegraphics[scale = 0.8]{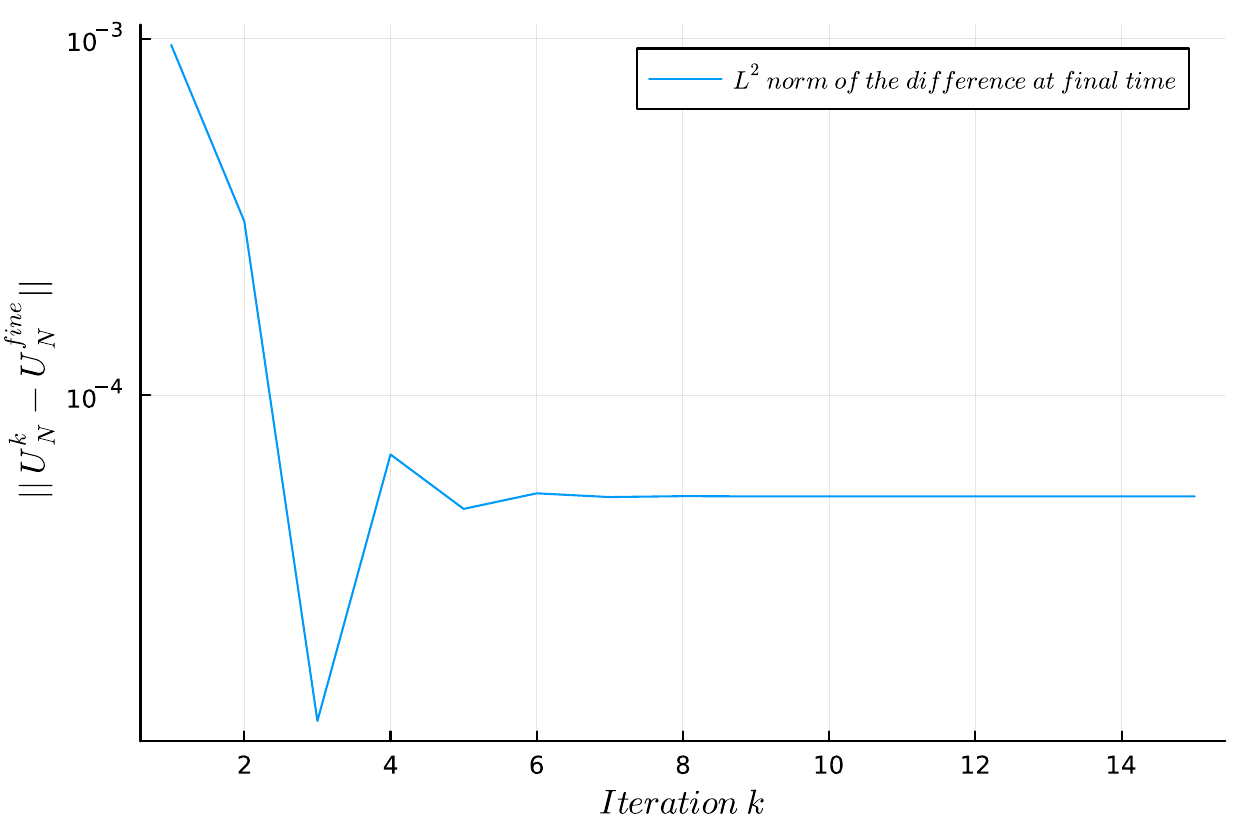}
	\caption{The difference $\|U^k_{N} - U^{fine}_N\|$ with respect to the iteration count $k$. Here, $N_t = 2^8$, $N = 2^4$, $M = 2^5$, $\alpha = 0.5$. }
	\label{fig:errorIterations}
\end{figure} 

The next test we perform is the evaluation of the efficiency against the sequential fine solver. The main parameter investigated here is the number of fine degrees of freedom $N_t \times M$ (that is, the reciprocal of the fine step $\Delta t$). The tolerance for stopping parareal iterations was set at $10^{-10}$. In Fig. \ref{fig:Runtime}, we can see both the time of computations (in seconds, evaluated after taking several iterations to reduce the computer background noise) and the \textit{speedup}, that is the ratio
\begin{equation}
	\text{speedup } := \frac{\text{fine solver time}}{\text{parareal time}}.
\end{equation}
As we can see, the computation time for both of our methods increases as a power law (note the log-log scale). Fine solver shows \textit{superquadratic} behavior while the parareal implementation is just slightly above \textit{linear} dynamics. This makes a profound difference in the case where the number of degrees of freedom is large. The lower plot in Fig. \ref{fig:Runtime} further stresses it. It shows that the speedup of the parareal algorithm increases linearly with the degrees of freedom reaching high values of $14-15$ times faster computations for $N_t \times M \approx 3 \times 10^4$. Even with a more moderate number of time steps, such as $O(10^3)$, we obtain a speedup of three to five times! This clearly shows that using our parareal scheme is very beneficial. 

\begin{figure}
	\centering
	\includegraphics[scale = 0.8]{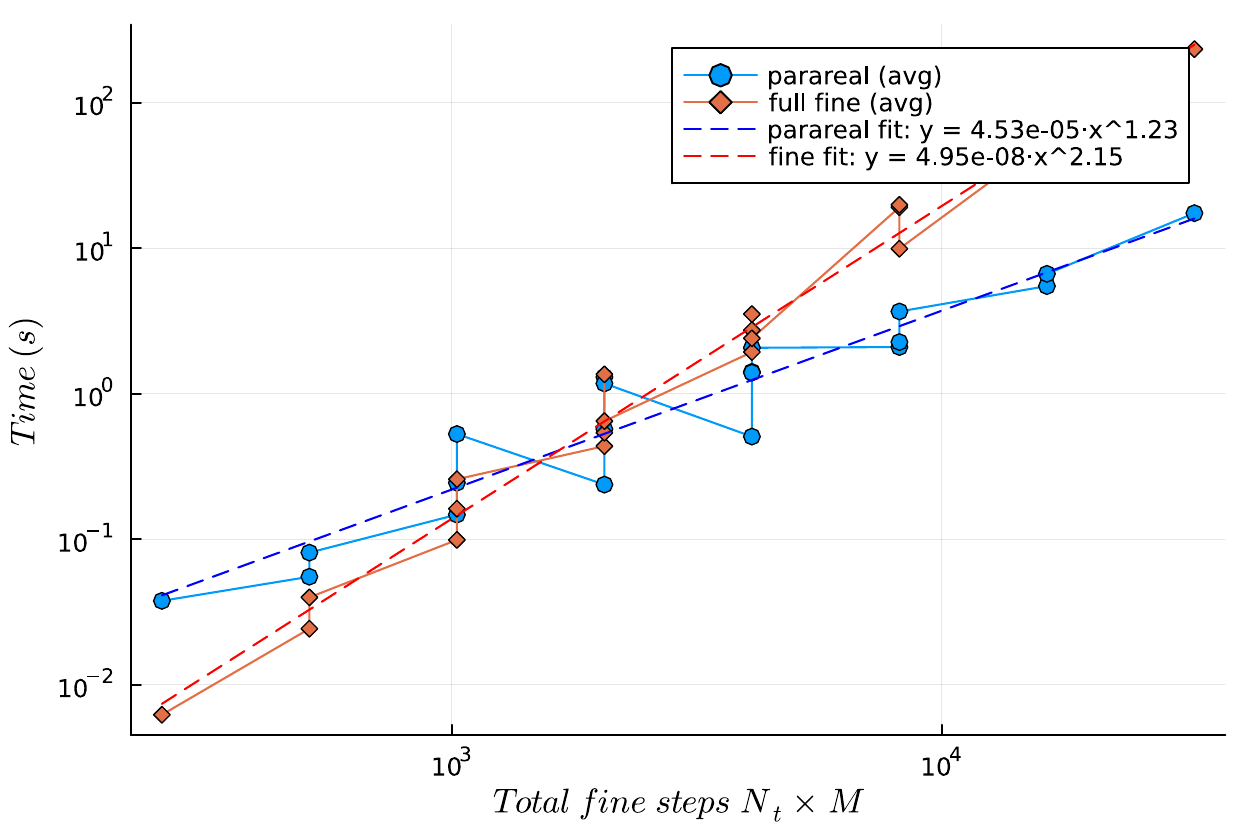}
	\includegraphics[scale = 0.8]{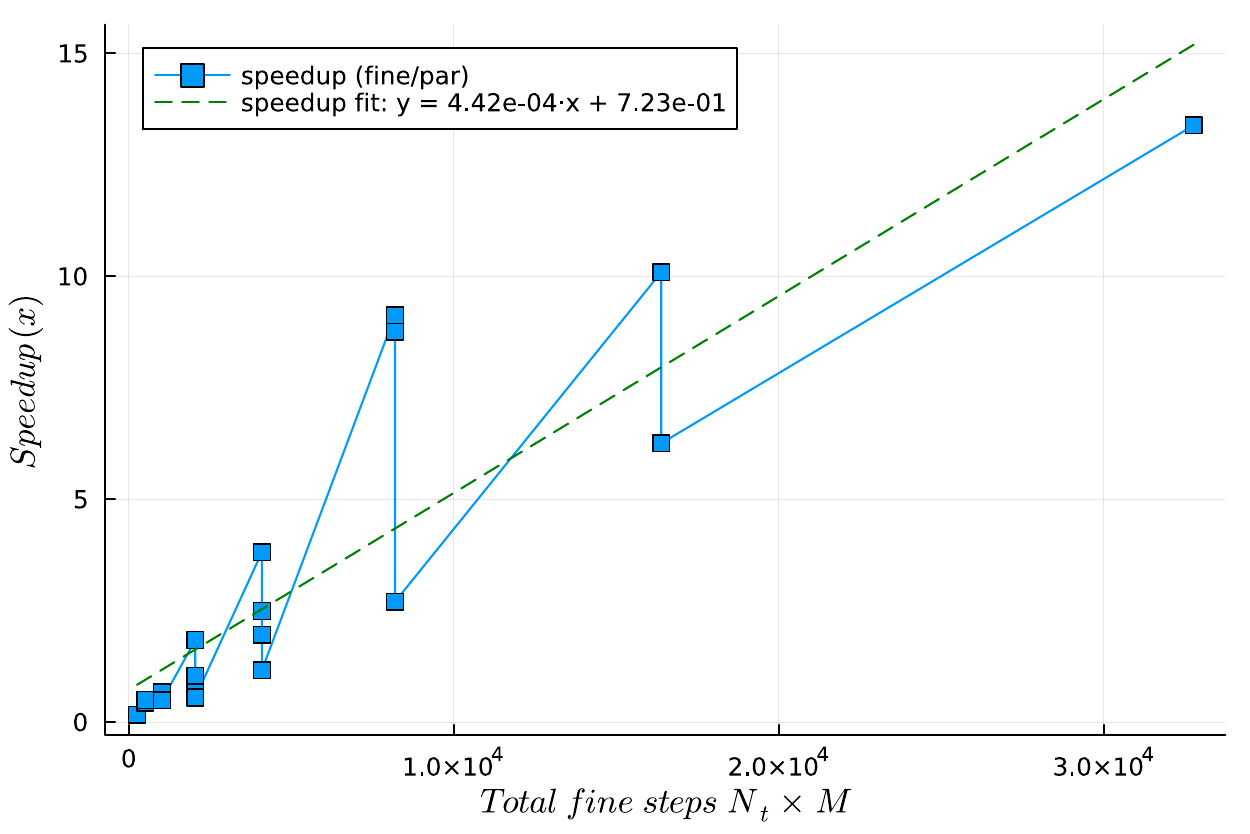}
	\caption{The runtime and the speedup for computations with the fine sequential solver and the parareal algorithm. $N = 2^4$, $\alpha = 0.5$. }
	\label{fig:Runtime}
\end{figure} 

The last test we conducted was the memory allocation check. In Fig. \ref{fig:Allocations}, we can see the amount of memory needed to perform simulations for both the parareal and the fine sequential solver. Although for a small number of degrees of freedom, the sequential integrator needs less memory, for larger ones the parareal develops a significant advantage requiring much fewer bytes for allocating memory. 

\begin{figure}
	\centering
	\includegraphics[scale = 0.8]{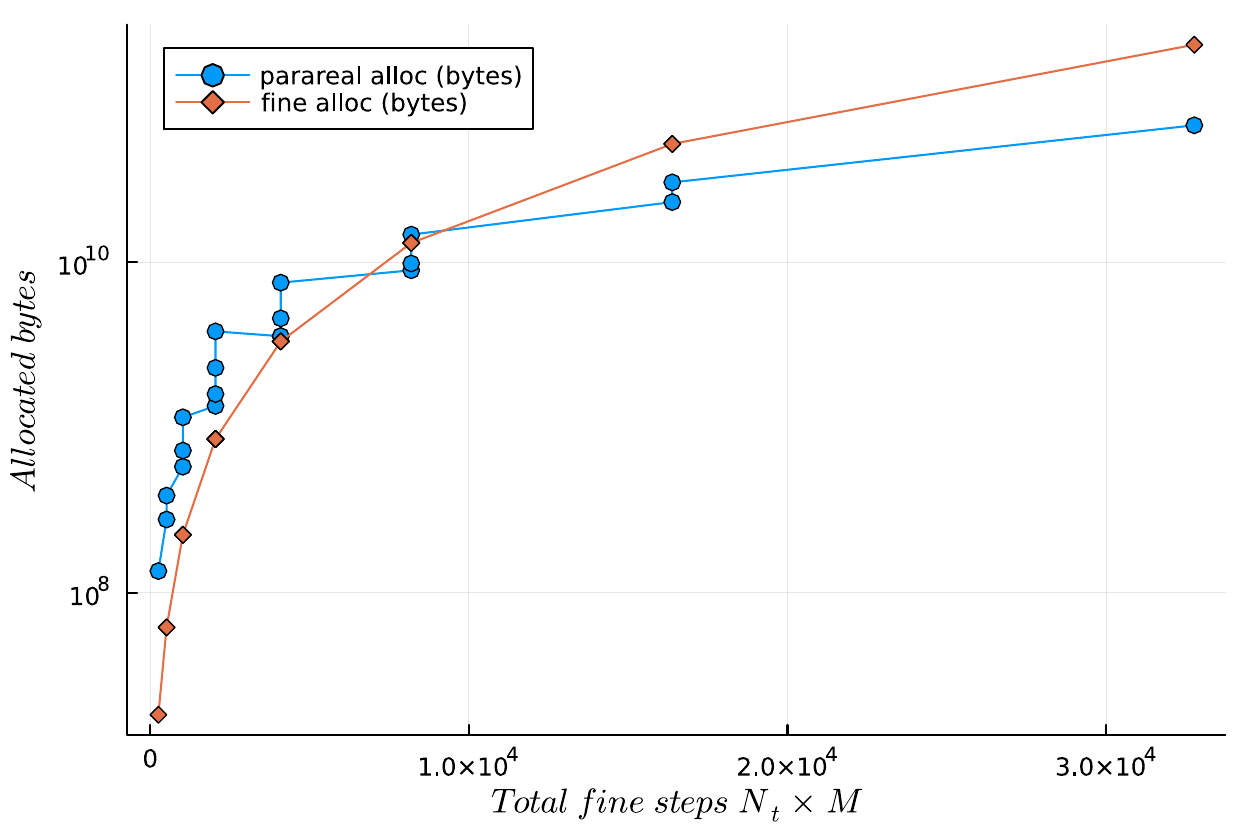}
	\caption{Number of allocated bytes for computations with the fine sequential solver and the parareal algorithm. Here, $N = 2^4$, $\alpha = 0.5$. }
	\label{fig:Allocations}
\end{figure} 

\section{Conclusion}
We have developed and analyzed a fully discrete solver for quasilinear time-fractional diffusion problems that couples the parareal-in-time algorithm with an L1 approximation of the Caputo derivative and a spectral collocation discretisation in space. Our main theoretical contribution is a rigorous convergence proof showing that parareal iterations converge to the serial L1–spectral solution in a finite number of steps and produce uniform bounds in the fractional exponent \(0<\alpha\le 1\). The spectral spatial discretisation delivers exponential accuracy for smooth solutions, while the parareal structure provides a practical speedup. Numerical experiments confirm the theoretical convergence and show substantial acceleration (often an order-of-magnitude reduction versus a sequential fine solver) with a speedup roughly linear in the work per coarse interval. 

For future work, we plan for: (i) accelerating long-time or large-NF runs by fast/convolution-compression techniques (sum-of-exponentials, convolution quadrature), (ii) investigating higher-order time integrators or adaptive coarse–fine strategies within parareal, and (iii) improving linear-solve robustness via spectral diagonalization where possible, preconditioning, iterative refinement, or GPU acceleration to extend the method to larger multidimensional problems. These extensions should preserve the accuracy-scalability balance of the method and broaden its applicability to challenging anomalous diffusion applications.

\section*{Acknowledgment}
Ł.P. is supported by the Polish National Agency for Academic Exchange (NAWA) under the Bekker Programme with the signature BPN/BEK/2024/1/00002. J.C. and K. S. are partially supported by the project PID2023-148028NB-I00.


\end{document}